\documentclass[12pt]{amsart}

\usepackage[margin=1in]{geometry}

\usepackage{amsmath,amscd,amssymb,amsfonts,latexsym,wasysym, mathrsfs, mathtools,hhline,color}
\usepackage[all, cmtip]{xy}
\usepackage{csquotes}
\usepackage{url}
\usepackage{comment}
\usepackage{qbordermatrix}
\definecolor{hot}{RGB}{65,105,225}

\usepackage[pagebackref=true,colorlinks=true, linkcolor=hot ,  citecolor=hot, urlcolor=hot]{hyperref}
\usepackage{ textcomp }
\usepackage{ tipa }
\usepackage{graphicx,enumerate}

\usepackage{geometry}
\theoremstyle{plain}
\newtheorem{theorem}{Theorem}[section]
\newtheorem{prop}[theorem]{Proposition}

\newtheorem{lm}[theorem]{Lemma}

\newtheorem{cor}[theorem]{Corollary}

\newtheorem{thrm}[theorem]{Theorem}

\theoremstyle{definition}

\newtheorem{defn}[theorem]{Definition}
\newtheorem{que}[theorem]{Question}
\newtheorem{rmk}[theorem]{Remark}

\newtheorem{assumption}[theorem]{Assumption}
\newtheorem{ex}[theorem]{Example}
\newtheorem*{ex*}{Example}

\def\be{\begin{equation}}
\def\ee{\end{equation}}

\def\bt{\begin{thrm}}
\def\et{\end{thrm}}

\def\bc{\begin{cor}}
\def\ec{\end{cor}}

\def\br{\begin{rmk}}
\def\er{\end{rmk}}

\def\bp{\begin{prop}}
\def\ep{\end{prop}}

\def\bl{\begin{lm}}
\def\el{\end{lm}}

\def\bex{\begin{ex}}
\def\eex{\end{ex}}

\def\bd{\begin{defn}}
\def\ed{\end{defn}}

\newcommand{\CP}{\mathbb{P}}

\newcommand{\C}{\mathbb{C}}
\newcommand{\Z}{\mathbb{Z}}
\newcommand{\Hom}{\mathrm{Hom}}
\newcommand{\Q}{\mathbb{Q}}
\newcommand{\K}{\mathbb{K}}
\newcommand{\V}{\mathcal{V}}
\newcommand{\cha}{\mathrm{char}}
\newcommand{\Fp}{\mathbb{F}_p}
\newcommand{\p}{\overline{\mathbb{F}}_p}
\newcommand{\orb}{\mathrm{orb}}
\newcommand{\M}{\mathbf{M}}

\newcommand{\m}{\mathbf{m}}

\newcommand{\tor}{\mathrm{tor}}
\newcommand{\im}{\mathrm{im}}
\newcommand{\homo}{\mathrm{Hom}}
\newcommand{\sV}{\mathcal{V}}

\newcommand{\sA}{\mathcal{A}}

\newcommand{\rank}{\mathrm{rank }}

\newcommand{\lcm}{\mathrm{lcm}}

\newcommand{\bm}[1]{\mbox{\boldmath{$#1$}}}

\title[]{$L^2$-type invariants and cohomology jump loci  for complex smooth quasi-projective varieties}

\author{Fenglin Li}
\address{Fenglin Li: Zhejiang Financial College, No.118 Xueyuan Street, Hangzhou, Zhejiang, 310018, P. R. China}
\email{fenglin125@126.com}

\author{Yongqiang Liu}
\address{Yongqiang Liu: The Institute of Geometry and Physics, University of Science and Technology of China, 96 Jinzhai Road, Hefei Anhui 230026, P. R. China} 
\email{liuyq@ustc.edu.cn}

\date{\today}

\keywords{Alexander module, Mahler measure, jump loci, orbifold map, $L^2$-Betti number, hyperplane arrangement, Milnor fiber}

\subjclass[2010]{14F45, 32S20, 32S22, 55N10; 14H30, 32S55}

\begin{document}
\maketitle
\begin{abstract}  
Let $X$ be a complex smooth quasi-projective variety with a fixed epimorphism $\nu\colon\pi_1(X)\twoheadrightarrow \Z$. 
We  consider the asymptotic behaviour of invariants such
as Betti numbers with any field coefficients and the order
of the torsion part of singular integral homology associated to $\nu$, known as the $L^2$-type invariants.  
At homological degree one, we give  concrete formulas for these limits by the geometric information of $X$ when $\nu$ is orbifold effective. The proof relies on  a study about cohomological degree one jump loci of $X$. We extend part of Arapura's result for cohomological degree one jump loci of $X$ with complex  field coefficients to the one with positive characteristic field coefficients. 
As an application, when $X$ is a hyperplane arrangement complement, a combinatoric upper bound is given for the number of parallel positive dimensional components in cohomological degree one jump loci with complex coefficients.
 Another application is that we give a positive answer to a question posed by Denham and Suciu for hyperplane arrangement. 
 \end{abstract}

\section{Introduction}
There is a general principal to consider a classical invariant of a finite CW complex and define its analog for its  covering space. This leads to the $L^2$-type invariants.  Atiyah  \cite{Ati76} introduced the notion of $L^2$-Betti numbers in the context of a regular covering  of a closed Riemannian manifold. After that,  there has been vast literature for the $L^2$-invariant theory, see \cite{Luc02}. The $L^2$-invariant theory has attracted considerable interests and has relations with many other fields, such as operator theory and Algebraic $K$-theory.   A particular important result is L\"uck's approximation theorem \cite{Luc94}, which states  that the $L^2$-Betti numbers of the universal cover of a finite polyhedron can be found as limits of normalised Betti numbers of finitely sheeted normal coverings.

\subsection{$L^2$-type invariants} \label{sec 1.1}
Assume that $X$ is a connected finite CW-complex with a fixed group epimorphism $\nu\colon \pi_{1}(X)\twoheadrightarrow\mathbb{Z}$. 
Let $X^\nu$ denote the corresponding infinite cyclic covering space of $X$.
For any positive integer $N$, let $X^{\nu,N}$ denote the covering space of $X$ associated to the following composition of maps
 $$\pi_{1}(X)\overset{\nu}{\twoheadrightarrow}\mathbb{Z}\twoheadrightarrow\mathbb{Z}/N\mathbb{Z}.$$ 
One can study $X^\nu$ by its approximation $X^{\nu,N}$. Consider the following two limits 
\be \label{1}  \lim_{N\rightarrow\infty}\frac{\dim H_{i}(X^{\nu,N},\mathbb{K})}{N} \ee
for any field coefficient $\K$ and
 \be \label{2}
 \lim_{N\rightarrow\infty}\frac{\log|H_{i}(X^{\nu,N},\mathbb{Z})_{\mathrm{tor}}|}{N},
 \ee
 where $ |H_{i}(X^{\nu, N},\mathbb{Z})_{\mathrm{tor}}|$ denotes the cardinality of the torsion part of $H_i(X^{\nu,N},\Z)$. 

It is known that the first limit always exists for any field coefficients and does not depend on the choice of chains, see e.g. \cite[Theorem 0.2]{LLS11}.
When $\cha(\K)=0$, the existence of limit follows also from the well known L\"{u}ck's  approximation theorem \cite{Luc94}, which shows that the first limit coincides with the  $L^2$-Betti number  defined by the Von Neumann algebra. (L\"{u}ck's approximation theorem is  proved with much more generality.)
Set
$$\alpha_i(X^\nu ,\K) \coloneqq\lim_{N\rightarrow\infty}\frac{\dim H_{i}(X^{\nu,N},\mathbb{K})}{N} $$
and one can compute it by homology jump loci as follows, see e.g. \cite[Theorem 2.5]{DS14}. 
\bp \label{prop introduction betti number} Let $\K$ be an algebraically closed field. 
With the assumptions and notations  as above,  for any $i\geq 0$ we have 
\begin{center}
$ \alpha_i(X^\nu,\K)=\dim H_i(X, L_\rho)$ for $\rho \in \K^*$ being general.
\end{center}
 Here $L_\rho$ is the rank one $\K$-local system on $X$ corresponding to $\rho$ pulling back by $\nu$.
In particular, $\alpha_i(X^\nu,\K)$ is always an integer.
\ep 
 We will give a self-contained proof for the above proposition, see  Proposition \ref{prop betti number}.  When $X$ is a complex smooth projective variety and $\K=\C$, similar formulas for Betti numbers and Hodge numbers can be found in \cite{CL19}. When $X$ is relaxed to be a complex smooth quasi-projective variety, there are a lot of works to study the polynomial periodicity for the Betti numbers and Hodge numbers of congruence covers, see \cite[Section 5]{Suc01} and \cite{Bud09}. However we do not pursue this direction in this paper.

On the other hand, the  limit in (\ref{2}) also exists, e.g. see \cite[Theorem 5]{Le14}, and it can be computed by the Mahler measure (see Definition \ref{def MM}) of the $i$-th integral  Alexander polynomial $\Delta_i(X^\nu) \in \Z[t,t^{-1}]$ (see Definition \ref{def Alexander}). 
Set $$\M_i(X^\nu)\coloneqq \lim_{N\rightarrow\infty}\frac{\log|H_{i}(X^{\nu,N},\mathbb{Z})_{\mathrm{tor}}|}{N}.$$

 The above discussion shows that these two limits  are determined by jump loci and integral Alexander polynomials.  When $X$ is a smooth complex quasi-projective variety, under certain assumptions for $\nu$  we give concrete formulas to compute these limits at homological degree one (namely $i=1$) by  geometric information of $X$. The formulas rely on a detailed study about cohomological degree  one jump loci of $X$.

\subsection{Main results}

\subsubsection{Cohomology jump loci and orbifold maps} \label{sec 1.2.1}
 Let $X$ be a smooth complex quasi-projective variety with $\pi_1(X)=G$ and $\K$ be an algebraically closed field.
The group of $\K$-valued  characters, $ \homo(G,\K^*)$, is a commutative affine algebraic group. Each character $\rho \in \homo(G,\K^*)$ defines a rank one $\K$-local system on $X$, denoted by $L_{\rho}$.
\bd \label{def jump loci}
The cohomology jump loci of $X$ are defined as
$$\sV^i_k(X,\K)\coloneqq \lbrace \rho\in \homo(G,\K^*) \mid \dim_{\K} H^{i}(X, L_{\rho})\geq k \rbrace.$$ 
When $k=1$, we simply write $\sV^i(X,\K)$.
\ed

When $\K=\C$, the cohomology jump loci of complex smooth quasi-projective variety  have been  intensively studied. In particular, the well-known structure theorem (see e.g. Theorem \ref{structure theorem}) for $\sV^i_k(X,\C)$ put strong constraints for the homotopy type of complex smooth quasi-projective variety.  If we focus on  cohomological degree one jump loci, a  celebrity result  due to Beauville \cite{Bea}, Arapura \cite{Ara97} and Artal Bartolo, Cogolludo-Agust\'in and Matei \cite{ACM13} puts  even stronger constraints for its fundamental group.

To explain their results, we recall the definition of orbifold maps. Let  $\Sigma_{g,r}$ be a complex smooth  algebraic curve of genus $g\geq 0$ with $r\geq 0$ points removed. We always assume that $b_1(\Sigma_{g,r}) = 2g+r-1>0$, i.e., $\Sigma_{g,r} \neq \mathbb{CP}^1,\C$.
\begin{defn} \label{def orbifold map}
Let $X$ be a complex smooth quasi-projective variety.
An algebraic map $f:X \to \Sigma_{g,r}$  is called an orbifold map, if $f$ is surjective and has connected generic fiber. There exists a maximal Zariski open subset $U\subset \Sigma_{g,r}$ such that $f$ is a fibration over $U$. Say $B=\Sigma_{g,r}-U$ (could be empty) has $s $ points, denoted by $\{q_1,\ldots,q_s\}$.   We assign
the multiplicity $\mu_j\geq 1$ of the fiber   $f^{*}q_j$  (the $\gcd$ of the coefficients of the divisor $f^* q_j$) to the point $q_j$. Then $f^{-1}(q_j)$ is called a multiple fiber of $f$ if $\mu_j>1$. 
Such orbifold map $f$ is called of type $(g,r,{\bm \mu})$ with ${\bm \mu}=(\mu_1,\ldots,\mu_s)$.  

We say $f$ has no multiple fiber  if $\prod_{j=1}^s \mu_j=1$.  
 $f$ is called   null if $2g+r-2=0$ and $\prod_{j=1}^s \mu_j=1$, otherwise $f$ is called hyperbolic.  
\end{defn}
The orbifold group  associated to these data is  defined as
\[ \pi_{1}^{\orb}(\Sigma_{g,r},{\bm \mu})\coloneqq \pi_{1}(\Sigma_{g,r}\backslash \{q_{1}, \ldots, q_{s}\})/ \langle \gamma_{j}^{\mu_{j}}=1 \text{ for all } 1\leq j \leq s\rangle, \]
where $\gamma_{j}$ is a meridian of $q_{j}$. 
 When $\Sigma_{g,r}$ is clear in the context, we simply write $\Sigma$. 
 An orbifold map $f\colon X\to \Sigma$ of type $(g,r,{\bm \mu})$ induces a surjective map on the fundamental groups  (see e.g. \cite[Lemma 3]{CKO}) 
 $$f_* \colon \pi_1(X) \twoheadrightarrow \pi_{1}^{\orb}(\Sigma_{g,r},{\bm \mu}),$$
which  induces an embedding 
$$\sV^1(\pi_{1}^{\orb}(\Sigma_{g,r},{\bm \mu}),\K) \to \sV^1(X,\K),$$
see e.g. \cite[Proposition A.1]{Suc14B}.
The following theorem shows that every positive dimensional component of $\sV^1(X,\C)$ can be realized by the orbifold map in this way. This idea first appeared in Beauville’s work \cite{Bea} for the projective case and was extended to the
quasi-projective case by Arapura \cite{Ara97}. Further properties were found by Dimca \cite{Dim07,Dim09},  Dimca, Papadima and Suciu \cite{DPS08,DPS09}, Artal Bartolo, Cogolludo-Agust\'in and Matei \cite{ACM13}, etc.    
We  recall the theorem here and adjust it for our needs. 
\bt \label{thm Ara} \cite{Ara97,ACM13} 
Let $X$ be a complex smooth quasi-projective variety. 
\begin{itemize}
\item[(a)] We have
\[
\V^1(X,\C)=\bigcup_{f }{f^*} \sV^1(\pi_1^{\orb}(\Sigma_{g,r},{\bm \mu})),\C)\cup Z,
\]
where $Z$ is a finite set of torsion points and the union runs over a finite set of hyperbolic orbifold maps $f\colon  X\rightarrow \Sigma$ of of type $(g,r,{\bm \mu})$.   
\item[(b)]   Given an orbifold map $f\colon X\to \Sigma$ of type $(g,r,{\bm \mu})$ and any  $\rho \in \homo(\pi_1^{\orb}(\Sigma_{g,r},{\bm \mu}),\C^*)$, we have
$$\dim H^1(X,f^* L_\rho)\geq \dim H^1( \pi_1^{\orb}(\Sigma_{g,r},{\bm \mu}), L_\rho)$$
and equality holds with finitely many (torsion) exceptions. 
\end{itemize}
\et

In this paper, we consider the  above theorem with arbitrary algebraically closed field coefficients.  Theorem \ref{thm Ara}(a) 
is generalized by Delzant  as follows using Bieri-Neumann-Strebel invariants when $X$ is smooth projective.
\bt \label{thm Delz}\cite{Delz} Let $X$ be a complex smooth projective variety   and $\K$ be an algebraically closed field. Then   we have
\[
\V^1(X,\K)=\bigcup_{f }{f^*} \sV^1(\pi_1^{\orb}(\Sigma_{g,r},{\bm \mu})),\K)\cup Z',
\]
where $Z'$ is a finite set of points (depends on $\K$) and the union  runs over a finite set of hyperbolic orbifold maps   $f\colon  X\rightarrow \Sigma$ of type $(g,r,{\bm \mu})$.   
 \et 

We prove the following generalization of  Theorem \ref{thm Ara}(b). 

\bt \label{thm jump loci} Let $X$ be a complex smooth quasi-projective variety  and $\K$ be an algebraically closed field. Given an orbifold map $f\colon X\to \Sigma$ of type $(g,r,{\bm \mu})$ and any  $\rho \in \homo(\pi_1^{\orb}(\Sigma_{g,r},{\bm \mu}),\K^*)$, 
we have 
$$  \dim H^1(X,f^* L_\rho)\geq \dim H^1(\pi_1^{\orb}(\Sigma_{g,r},{\bm \mu}),L_\rho)$$
and equality holds with finitely many exceptions. 
\et 
 For any  $\rho \in \homo(\pi_1^{\orb}(\Sigma_{g,r},{\bm \mu}),\K^*)$, $\dim H^1(\pi_1^{\orb}(\Sigma_{g,r},{\bm \mu}),L_\rho)$ is computed in subsection \ref{section orbifold groups}. Theorem \ref{thm jump loci} together with Theorem \ref{thm Delz} indeed gives a complete description of all positive dimensional component of $\sV^1_k(X,\K)$ for all $k$ when $X$ is smooth projective.

\br 
 Theorem \ref{thm Ara} (see e.g. \cite[Theorem 9.3]{Suc14}) and Theorem \ref{thm Delz} (\cite{Del}) also work for compact K\"ahler manifold. Correspondingly, one can also prove  Theorem \ref{thm jump loci} for compact K\"ahler manifold by the same proof presented  in this paper.
\er

\subsubsection{Formulas for $L^2$-type invariants}
Next we connect the epimorphism $\nu$ to the orbifold maps.
\bd \label{def orbifold effective}
Let $X$ be a complex smooth quasi-projective variety with an epimorphism $\nu\colon \pi_1(X)\twoheadrightarrow\Z$. We say that $\nu$ is orbifold effective if there is an orbifold map $f\colon X\rightarrow\Sigma$ of type $(g,r,{\bm \mu})$  such that $\nu$ factors through $f_*$ as follows:
\[\xymatrix{
\pi_1(X)\ar@{->>}"1,3"^{\nu}\ar@{->>}[dr]_{f_*} & & \Z, \\
 & \pi_1^{\orb}(\Sigma_{g,r}, {\bm \mu})\ar@{->>}[ur] &
}\]
 We  say that $\nu$ is  orbifold effective by $f$ and call $\nu$ being of type $(g,r,{\bm \mu})$.
\ed

\br \label{remark}  
This definition is well-defined, see Remark \ref{rem well-defined}.
Note that there are examples where $\nu$ is not orbifold effective, see Example \ref{ex non orbifold effective}. On the other hand,  if $H^1(X,\Q)$ is a pure Hodge structure of type $(1, 1)$, (i.e. $X$ has a smooth compactification $\overline{X}$ such that $H^1(\overline{X},\Q)=0$), then any epimorphism  $\nu\colon \pi_1(X)\twoheadrightarrow\Z$ is orbifold effective. This is why we include null orbifold maps in Definition \ref{def orbifold map}. 
Typical examples are hyperplane arrangement complement, see Remark \ref{rem Hodge}.
\er

\bt \label{thm L2}
Let $X$ be a complex smooth quasi-projective variety and $\K$ be a field with $\cha(\K)=p\geq 0$. 
Consider an epimorphism $\nu\colon \pi_1(X)\twoheadrightarrow \Z$. Then we have the following.
\begin{itemize}
\item[(a)] If $\nu$ is orbifold effective of type $(g,r, {\bm \mu})$, we have
\begin{equation*}
\alpha_1(X^\nu,\K)=  2g+r-2+ \#\{ j  \mid p \text{ divides } \mu_j\}
\end{equation*}  
(here we use the convention that $0$ does not divide any nonzero integer) and 
\begin{equation*}
\M_1(X^\nu)=\sum_{j=1}^s \log \mu_j.
\end{equation*}   
In particular, $\Delta_1(X^\nu)$ is a product of $ \prod_{j=1}^s \mu_j$ with some cyclotomic polynomials.
\item[(b)] Assume that $X$ is either smooth projective or  $H^1(X,\Q)$ is a pure Hodge structure of type $(1, 1)$. 
If  $\alpha_1(X^\nu,\K)>0$ for some field $\K$ or $\M_1(X^\nu)>0$, $\nu$ is orbifold effective by a hyperbolic orbifold map.
\end{itemize}  
\et

The above formula  for $\alpha_1(X^\nu,\K)$ follows from Theorem \ref{thm jump loci} directly and the one for $\M_1(X^\nu)$ follows by a similar proof to Theorem \ref{thm jump loci}. Theorem \ref{thm L2}(b) follows from Theorem \ref{thm Delz} and Remark \ref{remark}.

\subsection{Application to hyperplane arrangements}
Let $X$ be the complement of a hyperplane arrangement in $\CP^n$.
Given an epimorphism $\nu\colon\pi_{1}(X)\rightarrow\Z$ and a field $\K$, we denote $\nu_{\Z}$ and $\nu_\K$ the corresponding element in $H^{1}(X,\Z)$ and $H^1(X,\K)$, respectively. Note that $\nu_{\Z}\cup \nu_{\Z}=0,$ hence   $\nu_{\K}\cup \nu_{\K}=0$. 
Then we get two chain complexes by cup product
\[ \xymatrix{
(H^{\ast}(X,  \Z) , \cdot\nu_{\Z})\colon & H^{0}(X,\Z)\ar[r]^{\nu_{\Z}} & H^{1}(X,\Z)\ar[r]^{\nu_{\Z}} & H^{2}(X,\Z)\ar[r] & \cdots \\
(H^{\ast}(X,\K),\cdot \nu_{\K})\colon & H^{0}(X,\K)\ar[r]^{\nu_{\K}} & H^{1}(X,\K)\ar[r]^{\nu_{\K}} & H^{2}(X,\K)\ar[r] & \cdots
} .\]
\bd 
We define {\it  the $i$-th Aomoto Betti number with $\K$-coefficients} as
\[ \beta_{i}(X,\nu_{\K})\coloneqq\dim_{\K}H^{i}(H^{\ast}(X,\K), \cdot\nu_{\K}). \]
and {\it  the $i$-th Aomoto torsion number} as
\[ \tau_i(X,\nu_\Z)\coloneqq | H^{i+1}(H^*(X,\Z), \cdot\nu_\Z)_{\tor}|. \]
Here the shift by 1 is due to the Universal Coefficient Theorem.
\ed
Since it is well-known that the cohomology ring $H^*(X,\Z)$ are combinatorially determined, so are  $\beta_*(X,\nu_\K)$ and $\tau_*(X,\nu_\Z)$ once $\nu$, considered as an element in $H^1(X,\Z)$, is fixed. 
 For any field $\K$, Papadima and Suciu \cite{PS10} showed that $$  \alpha_i(X^\nu,\K) \leq \beta_i(X,\nu_\K).$$
On the other hand, we give the following combinatorial upper bound for  $\M_*(X^\nu)$.
\bp \label{prop comb} Let $X$ be the complement of a hyperplane arrangement.
For any epimorphism $\nu\colon \pi_1(X) \to \Z$ and any $i\geq 0$, we have 
 $$ \exp(\M_i(X^\nu)) \mid \tau_i(X, \nu_\Z).$$ In particular, if $\nu$ is orbifold effective of type $(g,r,{\bm \mu})$, then \be \label{upper bound for mu}
  \prod_{j=1}^s \mu_j \mid \tau_1(X,\nu_\Z). \ee
\ep 

Based on Theorem \ref{thm Ara}(a) and the computations given in subsection \ref{section orbifold groups} later, we give a different interpretation of (\ref{upper bound for mu}) as follows. 
\bc Let $X$ be a hyperplane arrangement complement.  Let $V $ be a  positive dimensional irreducible  component of $\sV^1(X,\C)$ and say it has $m$ many parallel components (including itself) in $\sV^1(X,\C)$. 
An epimorphism $\nu\colon \pi_1(X) \to \Z$ induces an embedding
$\nu^*\colon \C^*\to \Hom(\pi_1(X),\C^*) .$
If there exists some $\rho\in \homo(\pi_1(X),\C^*)$ such that 
$ \im \nu^* \subset \rho \cdot V$, we have the following:
\begin{itemize}
\item If $\dim V\geq 2$, then $m$ divides  $\tau_1(X,\nu_\Z)$.
\item If $\dim V=1$, then $m+1$ divides $\tau_1(X,\nu_\Z)$.
\end{itemize} 
\ec 
\br 
It is a long-standing  open question for hyperplane arrangement complement $X$ if $\sV^1(X,\C)$ is  combinatorially determined.  
It is known that the positive dimension components of $\sV^1(X,\C)$ passing through the origin are combinatorially determined (see e.g. \cite[Section 3]{Suc14}) and can be described by the multinet structure due to Falk and Yuzvinsky \cite{FY07}. Then the above corollary gives a combinatorial upper bound  for the number of its parallel components, which is new up to our knowledge. 
Under certain assumptions for the multinet (see Assumption \ref{assumption}), we show that $\sV^1(X,\C)$ has no parallel two dimensional components associated to this multinet.

On the other hand, there is an approach, called pointed multinet structure, to find the translated positive dimensional component of $\sV^1(X,\C)$, which is  introduced by Suciu, e.g., see \cite[Definition]{DS14}. 
\er

Since $H_*(X,\Z)$ is always torsion free for hyperplane arrangement complement $X$, people was wondering if the Milnor fiber of a central hyperplane arrangement also exhibits the same property. In \cite{DS14}, Denham and Suciu provided a complete answer to this question by showing that for any prime number $p\geq 2$ there is a central hyperplane arrangement  whose Milnor fiber has non-trivial $p$-torsion in homology. Their results hold under the assumption that $p$ does not divide the number of hyperplanes 
and they asked if this assumption can be dropped \cite[Question 8.12]{DS14}. We
follow Denham and Suciu's construction in \cite{DS14} and answer their question positively. 

\bt \label{thm our DS} 
For every prime $p\geq 2$, there is a central hyperplane arrangement  whose Milnor
fiber  has non-trivial $p$-torsion in homology and $p$ divides  the number of hyperplanes in the arrangement. 
\et

\subsection{Structure of the paper}
In section 2, we introduce the Alexander modules and Alexander polynomials for the pair $(X,\nu)$ and study their relations with the $L^2$-type invariants. Section 3 is devoted to the cohomology jump loci and detailed calculations of these invariants for the orbifold group. In section 4, we prove Theorem \ref{thm jump loci} and Theorem \ref{thm L2}. Section 5 is devoted to hyperplane arrangements. In section 5.1, we prove Proposition \ref{prop comb} and compute some examples; while in section 5.2, we show that if a multinet of line arrangement admits certain kind of combinatoric structure,  then the corresponding orbifold map has no mutliple fiber. 
In section 5.3, we prove Theorem \ref{thm our DS}.

\subsection{Notations}
\begin{enumerate}
\item $\K$ is assumed to be an algebraically closed field in the rest of paper and $\cha(\K)$ denotes the characteristic of $\K$.
\item For any prime number $p\geq 2$, let $\p$ denote the algebraic closure of the finite field $\mathbb{F}_p $ with $p$ elements.
\item Let $\Z_{>0}$ and $\Z_{\geq 0}$ denote the collection of positive integers and non-negative integers, respectively.
\end{enumerate}

\textbf{Acknowledgments.} We thank Laurentiu Maxim, Botong Wang and Zhenjian Wang for valuable discussions.
The second named author is  partially supported  by National Key Research and Development Project SQ2020YFA070080, NSFC grant No. 12001511, the Project of Stable Support for Youth Team in Basic Research Field, CAS (YSBR-001),  the project ``Analysis and Geometry on Bundles" of Ministry of Science and Technology of the People's Republic of China and  Fundamental Research Funds for the Central Universities.

\section{Alexander and $L^2$-type invariants}
In this section, we always assume that  $X$ is a connected finite CW-complex with a fixed group epimorphism $\nu\colon\pi_{1}(X)\twoheadrightarrow\mathbb{Z}$. 
Denote by $X^\nu$ and $X^{\nu,N}$ the corresponding covering spaces of $X$ as introduced in Section \ref{sec 1.1}. We study various topological properties of $X^\nu$ and its approximation $X^{\nu,N}$.

\subsection{Algebraic Preliminaries}
We start with some basic commutative algebra facts.
Let $R$ be a Noetherian UFD and $M$ be a finitely generated $R$-module. The rank of $M$, denoted by $\rank M$, is defined to be $\dim_{Q}Q\otimes_{R}M$, where $Q$ is the fraction field of $R$.
Consider a finite presentation of $M$
\[ \xymatrix{
R^{p}\ar[r]^{\partial} & R^{q}\ar[r] & M\ar[r] & 0,
} \]
where $\partial$ is a $(p\times q)$ matrix with entries in $R$. Assume that the matrix $\partial$ has rank $r$. Let
$I(M)$ denote the ideal generated by all possible $(r\times r)$-minors of $\partial$. 
Two finite presentations of $M$ can be related by a sequence of elementary operations, so this ideal does not depend on the choice of presentation. Since $R$ is a UFD, we can define the greatest common divisor of an ideal and set 
\[ \Delta(M)\coloneqq \mathrm{gcd}(I(M)). \]
Then $\Delta(M)$ is defined uniquely up to a multiplication with a unit of $R$. 
When $\partial=0$, $\Delta(M)=1$ by convention.
\par
\bl \cite[Lemma 4.9]{Tur01} \label{lemma tor}
With the assumptions and notations as above, we have
\[ \Delta(M)=\Delta(M_{\tor}),\]
where $M_\tor$ denote the torsion sub-module of $M$.
\el
\par

\bl \label{lem commutative algebra}
Let $0\to M'\to M \to M''\to 0$ be a short exact sequence of finitely generated $R$-modules. 
Then we have $$\Delta(M)\mid \Delta(M')\cdot \Delta(M'').$$
 Moreover, if  $\rank M=\rank M''$ (i.e. $\rank M'=0$), we have  $$\Delta(M)= \Delta(M')\cdot \Delta(M'').$$ 
\el
\begin{proof}
The first claim follows from Horseshoe Lemma \cite[Horseshoe Lemma 2.2.8]{Wei94}. In fact, for a finite presentation of $M'$ with matrix $\partial'$ and one for $M''$ with matrix $\partial''$, Horseshoe Lemma gives a way to construct a finite presentation for $M$ with matrix \begin{center}
$\begin{pmatrix}
\partial' & *  \\
0  & \partial'' 
\end{pmatrix}.$
\end{center}
Since localization is an exact functor, we get  $\rank M=\rank M'+\rank M''$, which implies \begin{center}
$\rank\begin{pmatrix}
\partial' & *  \\
0  & \partial'' 
\end{pmatrix} = \rank \partial'+\rank \partial''.$
\end{center}
Then it is easy to see $\Delta(M)\mid \Delta(M')\cdot \Delta(M'').$

For the second claim, we have the following commutative diagram:
$$\xymatrix{
M \ar[r] \ar[d] & M'' \ar[d] \\
M\otimes_R Q \ar[r] & M''\otimes_R Q
} $$
where $Q$ is the fraction field of $R$. The additional assumption $\rank M= \rank M''$ implies that the bottom horizontal map is an isomorphism. We claim that the induced map $M_{\tor} \to M''_{\tor}$ is surjective. For any $x\in M''_{\tor}$, there exists $y\in M$ maps to $x$. Note that $x=0$ in $M''\otimes_R Q$, hence so is $y$, which implies $y\in M_{\tor}$. Note that taking torsion part is a left exact functor. Putting all together we get a short exact sequence of finitely generated $R$-modules:
$$ 0\to M'_{\tor} \to M_{\tor} \to M''_{\tor}\to 0.$$
It implies $\Delta(M_{\tor})=\Delta(M'_{\tor})\cdot \Delta(M''_{\tor})$, see e.g. \cite[Lemma 5]{Lev67}. Then the claim follows from Lemma \ref{lemma tor} directly.  
\end{proof}

\subsection{Alexander modules and Alexander polynomials}
Recall that $X$ is a connected finite CW-complex with an epimorphism $\nu\colon\pi_{1}(X)\twoheadrightarrow\mathbb{Z}$. 
The group of covering transformations of $X^{\nu}$ is isomorphic to $\Z$ and acts on it.   By choosing lifts of the cells of $X$ to $X^{\nu}$, we obtain a free basis for the cellular chain complex (with $R$-coefficients) of $ X^\nu$ as   $R[t,t^{-1}]$-modules, where $R[t,t^{-1}]= R[\Z]$.   So the cellular chain complex of $X^{\nu}$, $C_{*}(X^{\nu}, R)$, is a bounded complex of finitely generated free $R[t,t^{-1}]$-modules:
\be \label{chain complex}
 \cdots  \to  C_{i+1}(X^\nu, R) \overset{\partial_{i}^{R}}{\to} C_i(X^\nu, R) \overset{\partial_{i-1}^{R}}{\to} C_{i-1}(X^\nu, R)  \overset{\partial_{i-2}^{R}}{\to}   \cdots \overset{\partial_0^{R}}{\to} C_0(X^\nu, R)  \to 0 .
\ee
With the above free basis for $C_*(X^\nu,R)$, $\partial^R_i$ can be written down as a matrix with entries in $R[t,t^{-1}]$. Note that $R[t,t^{-1}]$ is also a Notherian UFD. \par

\bd \label{def Alexander} The $i$-th homology group $H_i(X^\nu, R)$ of $C_{*}(X^{\nu},R)$, regarded as a finitely generated $R[t,t^{-1}]$-module, is called {\it the $i$-th  homology Alexander module of the pair $(X,\nu)$ with $R$-coefficients}. 
 $ \Delta(H_{i}(X^\nu,R))$  is called {\it the $i$-th Alexander polynomial of $(X,\nu)$  with $R$-coefficients},  denoted by $\Delta_{i}(X^\nu,R)$. 

Consider the case $R=\Z$.
Note that 
$\Delta_i(X^\nu,\Z)$ is only defined up to multiplication by a unit in $\Z[t,t^{-1}]$.  But there is a unique representative of the associate class with no negative powers of $t$,
 non-zero constant term and positive coefficient for the leading term. We simply use the notation $\Delta_i(X^\nu)$ to denote this representative. 
In particular,  $\Delta_i(X^\nu)$ is an integer valued  polynomial. 
\ed

Note that the homology Alexander modules and Alexander polynomials are homotopy invariants for the pair $(X,\nu)$.
 In particular, $\Delta_1(X^\nu,R)$  depends only on $\pi_1(X)$ and $\nu$. 
One can also define  $\Delta_i(X^\nu,R)$ from the map $\partial^R_i$ in (\ref{chain complex}) directly.
\bp \label{another def for Alexander polynomial}
Suppose that the $R[t,t^{-1}]$-module map $\partial_{i}^R$ in (\ref{chain complex}) has rank $r$. Then $\Delta_{i}(X^\nu,R)$ is equal to the greatest common divisor of all possible $(r\times r)$-minors of $\partial_{i}^R$.
\ep
\begin{proof}
Note that we have the following short exact sequence of $R[t,t^{-1}]$-modules
\[ 0\rightarrow H_i(X^\nu,R)\rightarrow C_i(X^\nu,R)/\im \partial_{i}^R\rightarrow \im \partial_{i-1}^R\rightarrow 0, \]
where $\im$ denotes the image functor.
Since $\mathrm{im}\partial_{i-1}^R$ is torsion free, we get 
\[ H_i(X^\nu,R)_{\tor}\cong(C_i(X^\nu,R)/\mathrm{im}\partial_i^R)_{\tor}. \]
Then Lemma \ref{lemma tor} implies that
\[ \Delta_i(X^\nu,R)=\Delta(H_i(X^\nu,R)_{\mathrm{tor}})=\Delta((C_i(X^\nu,R)/\mathrm{im}\partial_i^R)_{\mathrm{tor}})
=\Delta(C_i(X^\nu,R)/\mathrm{im} \partial_i^R). \]
Consider the finite presentation of $C_i(X^\nu,R)/\mathrm{im} \partial_i^R$  $$C_{i+1}(X^\nu, R) \overset{\partial_{i}^{R}}{\to} C_i(X^\nu, R)\to C_i(X^\nu,R)/\mathrm{im} \partial_i^R \to 0$$
Then the claim follows by definition.
\end{proof}

 \subsection{$L^2$-type invariants}
 \subsubsection{$L^2$-Betti number}
As mentioned in Section \ref{sec 1.1}, for any field $\K$  the following limit
$$ \alpha_i(X^\nu,\K)=\lim_{N\rightarrow\infty}\frac{\dim H_{i}(X^{\nu,N},\mathbb{K})}{N} $$
exists. Note that  Betti number with $\K$-coefficients only depend on  $\cha(\K)$, not on the specific choice of the field $\K$. So without loss of generality, we only need to consider the case where $\K$ is algebraically closed.
In this subsection, we give a self-contained proof for the existence of the limit and relate these limits with Alexander modules.

\bp \label{prop betti number}
For any $i\geq 0$ the limit $  \lim_{N\rightarrow\infty}\frac{\dim H_{i}(X^{\nu,N},\mathbb{K})}{N}$ exists. Moreover, we have $$ \alpha_i(X^\nu,\K)=\rank H_i(X^\nu,\K)=\dim H_i(X, L_\rho)$$ where $\rho \in \K^*$ is general and  $L_\rho$ is the corresponding rank one $\K$-local system on $X$ pulling back by $\nu$. 
\ep

For the proof of this proposition, we need the following lemma.
\bl \label{lem linear algebra}
Given a nonzero polynomial $h(t)=t^d+\cdots+ a_{1}t+a_0\in\K[t]$, we denote by $J_N$ the following   $N\times N$ ($N\geq d$) matrix 
\[ J_N\coloneqq\begin{pmatrix}
0 & 1 & 0 & \cdots & 0 \\
0 & 0 & 1 & \cdots & 0 \\
\vdots & \vdots & \vdots & & \vdots \\
0 & 0 & 0 & \cdots & 1 \\
1 & 0 & 0 & \cdots & 0 \\
\end{pmatrix}. \]
 Then there exists a positive integer $c$ such  that for $N$ large enough we have
\[ N-c\leq \mathrm{rank} h(J_N)\leq N, \]
where $c$ does not depend on $N$.
\el
\begin{proof}
Set $g(t)=\gcd(t^{N}-1,h(t))$. Then there exist $u_{1}(t),u_{2}(t)\in\K[t]$ such that $$g(t)=u_{1}(t)(t^{N}-1)+u_{2}(t)h(t).$$ Note that $J_N^N=\mathrm{Id}$, where $\mathrm{Id}$ is the $(N\times N)$ identity matrix.
It gives that $$g(J_N)=u_{1}(J_N)(J_N^{N}-\mathrm{Id})+u_{2}(J_N)h(J_N)=u_{2}(J_N)h(J_N),$$ hence $\mathrm{rank} g(J_N)\leq \mathrm{rank} h(J_N)$. 
On the other hand, since $g(t)\mid h(t)$, we have $\mathrm{rank} g(J_N)\geq \mathrm{rank} h(J_N)$. Hence
\[ \mathrm{rank} g(J_N)=\mathrm{rank}  h(J_N). \]

There exists a positive integer $c$, which does not depend on $N$, such that $g( t)\mid t^{c }-1.$ 
In fact, since $\K$ is algebraically closed, we can write down $h(t)= \prod_{j=1}^d (t-\alpha_j).$ Collect all $\alpha_j$ such that there exists some positive integer $n_j$ with $\alpha_j^{n_j}=1$. When $\cha(\K)=0$, $c$ can be taken as the product of such $n_j$; while when $\cha(\K)=p>0$, $c$ can be taken as the product of such $n_j$ with some power of $p$. 

Note that for any field $\K$ the dimension of $\mathrm{ker}(J^{c}_N- \mathrm{Id})$ is $\gcd(c,N)$. 
Hence for $N$ large enough we have
\[ \mathrm{rank}h(J_N)=\mathrm{rank}g(J_N)\geq\mathrm{rank}(J^{c}_N- \mathrm{Id})\geq N-c \]
The other part of the inequality is obvious.
\end{proof}

\begin{proof}
[Proof of Proposition \ref{prop betti number}]
Recall that $C_{*}(X^{\nu}, \K)$ is the cellular chain complex of $X^\nu$ with the field coefficient $\K$. Each $C_{i}(X^\nu,\K)$ is a finitely generated free $\K[t^{\pm 1}]$-module, since $X$ is a finite CW complex. 
Note that $\mathbb{K}[t^{\pm 1}]$ is a PID. By choosing suitable basis we can assume that $\partial_{i}^{\K}$ has the following form
\[ \begin{pmatrix}
h_{1} & & & & \\
 & \ddots & & & \\
 & & h_{m_i} & & \\
 & &  & 0 & 
\end{pmatrix}\]
with $m_i=\rank \partial_i^\K$. In particular, $$\rank H_i(X^\nu, \K)= \rank C_i(X^\nu,\K)-m_i-m_{i-1}. $$

By tensoring with $\K[t^{\pm 1}]/(t^{N}-1)$, we get the cellular chain complex for $X^{\nu,N}$ with $\K$-coefficients. We use $\partial_{i,N}^{\K}$ to denote the corresponding differential map. Then
\[ \dim H_{i}(X^{\nu,N},\K)=\mathrm{rank}C_{i}(X^\nu,\K)\cdot N -\mathrm{rank}\partial_{i,N}^{\K}-\mathrm{rank}\partial_{i-1,N}^{\K}. \]
Note that each polynomial $h_{j}(t)$ ($1\leq j \leq m_i)$) becomes a $N\times N$ matrix $h_{j}(J_N)$ in $ \partial_{i,N}^{\K}$. By Lemma \ref{lem linear algebra}, there exist constant numbers $c_{i}$ and $c_{i-1}$ such that 
\[ \rank H_i(X^\nu, \K)\leq\frac{\dim H_{i}(X^{\nu,N},\K)}{N}\leq \rank H_i(X^\nu, \K)+\frac{c_{i}+c_{i-1}}{N}.\]
Taking $N\rightarrow\infty$, we get 
\[ \alpha_{i}(X^\nu,\K)=\lim_{N\to \infty}\frac{\dim H_{i}(X^\nu,\K)}{N}=\rank H_i(X^\nu, \K). \]

\medskip

 Let $\pi_{N}\colon X^{\nu,N}\rightarrow X$ denote the covering map. Then  $$H_{i}(X^{\nu,N},\K)=H_{i}(X,(\pi_{N})_*\K),$$ where $(\pi_{N})_*\K$ is the push forward of the $\K$-constant sheaf on $X^{\nu,N}$, hence a rank $N$ $\K$-local system. Since we have already known that $\lim_{N\rightarrow\infty}\frac{\dim H_{i}(X^\nu,\K)}{N}$ exists, one can compute this limit with the subsequence $\{N \in \Z_{>0}\mid\cha(\K)\nmid N\}$. In this case $(\pi_{N})_*\K$ is a direct sum of $N$ many rank one local systems. 
 Since 
the function  $\K^* \to \Z $ by sending $\rho\in \K^*$ to $\dim H_{i}(X,L_{\rho})$ is semi-continuous,
  there are at most finitely many rank one  $\K$-local systems with larger homology dimension than the general one. In particular, these finitely many local systems only depends on the pair $(X,\nu)$, not on $N$. Hence we have
\[ \dim H_{i}(X,L_{\rho})\leq\frac{b_{i}(X^{\nu,N},\mathbb{K})}{N}\leq \dim H_{i}(X,L_{\rho})+\frac{c'}{N},\]
for some constant number $c'$ and $\rho \in \K^*$ being general. 
Taking $N\rightarrow\infty$, we are done.
\end{proof}

\subsubsection{Limits of torsion}
One can also study the following limit associated to  the order of the torsion part of $H_i(X^{\nu,N},\Z)$
$$\lim_{N\rightarrow\infty}\frac{\log|H_{i}(X^{\nu,N},\mathbb{Z})_{\mathrm{tor}}|}{N}.$$
This limit  can be computed by the Mahler measure of the Alexander polynomial $\Delta_i(X^\nu)$.
Let us  recall the definition of Mahler measure based on Jensen's formula \cite[p. 208]{Ahl66}.

\bd \label{def MM} \cite{Mah60,Mah62}
For a non-zero polynomial $h(t)\in \Z[t]$, assume that \begin{center}
$h(t)=a_{d}t^{d}+a_{d-1}t^{d-1}+\cdots+a_{0}=a_{d}\cdot\prod_{j=1}^{d}(t-\alpha_{j})$ with $a_d\neq 0$ and $\alpha_j\in \C$.
\end{center} 
One can define the {\it Mahler measure} of $h$ to be
\[ \M(h)=\log(|a_{d}| \cdot \prod_{j=1}^{d}\max\{1,|\alpha_{j}|\}). \]
We say that  $h$ is of {\it cyclotomic  type}, if all its roots $ \alpha_j$ are roots of unity. In particular, in this case $\M(h)= \log|a_{d}|$.
\ed

\begin{theorem}\cite[Theorem 5]{Le14}
With the above assumptions and notations, for any $i\geq 0$ the limit $\lim_{N\rightarrow\infty}\frac{\log|H_{i}(X^{\nu,N},\mathbb{Z})_{\mathrm{tor}}|}{N}$ exists and we have that
\[ \lim_{N\rightarrow\infty}\frac{\log|H_{i}(X^{\nu,N},\mathbb{Z})_{\mathrm{tor}}|}{N}=\M(\Delta_{i}(X^\nu)).\]
\end{theorem}

Based on the above theorem, we set
 $$ \M_i(X^\nu) \coloneqq \lim_{N\rightarrow\infty}\frac{\log|H_{i}(X^{\nu,N},\mathbb{Z})_{\mathrm{tor}}|}{N}=\M(\Delta_{i}(X^\nu)).$$

\subsubsection{Universal Coefficients Theorem}
By the Universal Coefficients Theorem, $$\dim H_i(X,\C) \leq \dim H_i(X,\K)   $$
for any field coefficient $\K$. Moreover, if $\cha(\K)=p>0$ and the inequality is strict, then either $H_i(X,\Z)$ or $H_{i-1}(X,\Z)$ has non-trivial $p$-torsion. On the other hand, if $H_i(X,\Z) $ has non-trivial $p$-torsion, then 
\begin{center}
$\dim H_i(X,\C) < \dim H_i(X,\p) $ and $\dim H_{i+1}(X,\C) < \dim H_{i+1}(X,\p) $
\end{center}

\bp \label{prop large enough}
Let $X$ be a connected finite CW-complex with a fixed group epimorphism $\nu\colon \pi_{1}(X)\twoheadrightarrow\mathbb{Z}$. Then we have  $\alpha_i(X^\nu,\C)\leq \alpha_i(X^\nu,\bar{\mathbb{F}}_p)$. Moreover,
if  $\alpha_i(X^\nu,\C)<\alpha_i(X^\nu,\bar{\mathbb{F}}_p)$, then we have the following
\begin{itemize}
\item[(1)] For $N$ large enough either $H_i(X^{\nu,N},\Z)$ or $H_{i-1}(X^{\nu,N},\Z)$ has non-trivial $p$-torsion.
\item[(2)] $ \Delta_i(X^\nu)=0 $ or $ \Delta_{i-1}(X^\nu)=0 $  in $\p[t,t^{-1}]$, hence $\M_i(X^\nu)>0$ or $\M_{i-1}(X^\nu)>0$, respectively.
\end{itemize} 
\ep
\begin{proof}
(1) Set $\epsilon=|\alpha_i(X,\C)-\alpha_i(X,\bar{\mathbb{F}}_p)|/2$. For $N$ large enough, we have
\[ \frac{\dim_{\C}H_i(X^{\nu,N},\C)}{N}<\alpha_i(X^\nu,\C)+\epsilon=\alpha_i(X^\nu,\bar{\mathbb{F}}_p)-\epsilon<\frac{\dim_{\bar{\mathbb{F}}_p}H_i(X^{\nu,N},\bar{\mathbb{F}}_p)}{N}. \]
Hence $\dim_{\C} H_i(X^{\nu,N},\C) < \dim_{\bar{\mathbb{F}}_p} H_i(X^{\nu,N},\bar{\mathbb{F}}_p)$ and the conclusion follows.

(2)  $\alpha_i(X^\nu,\C)<\alpha_i(X^\nu,\bar{\mathbb{F}}_p)$ is equivalent to
\[\mathrm{rank}C_i(X^\nu,\C)-\mathrm{rank}\partial_i^\C-\mathrm{rank}\partial_{i-1}^\C<\mathrm{rank}C_i(X^\nu,\bar{\mathbb{F}}_p)-\mathrm{rank}\partial_i^{\bar{\mathbb{F}}_p}-\mathrm{rank}\partial_{i-1}^{\bar{\mathbb{F}}_p},\]
which implies either $\mathrm{rank}\partial_i^\Z=\mathrm{rank}\partial_i^\C> \mathrm{rank}\partial_i^{\bar{\mathbb{F}}_p}$
or the same inequality for degree $i-1$.
Note that $\mathrm{rank}\partial_i^\Z> \mathrm{rank}\partial_i^{\bar{\mathbb{F}}_p}$ if and only if $p\mid \Delta_i(X^\nu)$ (i.e., $  \Delta_i(X^\nu)=0 $ in $\p[t,t^{-1}]$). 
In particular, $\M_i(X^\nu)\geq \log p$. Then the claim follows.
\end{proof}

\br \label{rem uct} If $\Delta_i(X^\nu) $ is of cyclotomic type and $\M_i(X^\nu) \neq 0$, then there exists a prime number $p$ such that $\Delta_i(X^\nu)=0$ in $\p[t,t^{-1}]$, hence 
 \begin{center}
 $\alpha_i(X^\nu,\C)<\alpha_i(X^\nu,\p)$ and $\alpha_{i+1}(X^\nu,\C)<\alpha_{i+1}(X^\nu,\p)$.
\end{center}  
\er

\subsection{Aomoto complex}
Given an epimorphism $\nu\colon\pi_{1}(X)\rightarrow\Z$, we denote $\nu_{\Z}$ the corresponding element in $H^{1}(X,\Z)$. By obstruction theory, there is a map $g\colon X\rightarrow S^{1}$ and a class $\omega\in H^{1}(S^{1},\Z)$ such that $\nu_{\Z}=g^{\ast}(\omega)$. Hence we have
\[ \nu_{\Z}\cup \nu_{\Z}=f^{\ast}(\omega\cup\omega)=0. \]
For any field $\K$, there is a corresponding class $\nu_{\K}\in H^{1}(X,\K)$ and we have $\nu_{\K}\cup \nu_{\K}=0$.
Then one gets the following  two Aomoto complexes by cup product
\[ \xymatrix{
(H^{\ast}(X,  \Z) , \cdot\nu_{\Z})\colon & H^{0}(X,\Z)\ar[r]^{\nu_{\Z}} & H^{1}(X,\Z)\ar[r]^{\nu_{\Z}} & H^{2}(X,\Z)\ar[r] & \cdots
} \]
and
\[ \xymatrix{
(H^{\ast}(X,\K),\cdot \nu_{\K})\colon & H^{0}(X,\K)\ar[r]^{\nu_{\K}} & H^{1}(X,\K)\ar[r]^{\nu_{\K}} & H^{2}(X,\K)\ar[r] & \cdots
} \]
\bd 
With the above assumptions and notations, we define {\it  the $i$-th Aomoto Betti number with $\K$-coefficients} as
\[ \beta_{i}(X,\nu_{\K})\coloneqq\dim_{\K}H^{i}(H^{\ast}(X,\K), \cdot\nu_{\K}). \]
and {\it  the $i$-th Aomoto torsion number} as
\[ \tau_i(X,\nu_\Z)\coloneqq | H^{i+1}(H^*(X,\Z), \cdot\nu_\Z)_{\tor}|. \]
Here the shift by 1 is due to the Universal Coefficient Theorem.
\ed

The following nice theorem due to Papadima-Suciu \cite{PS10} gives the relation between $\alpha_i(X^\nu,\K)$ and $\beta_i(X,\nu_\K)$.
In fact, they constructed a spectral sequence converges to $H_*(X^\nu,\K)$ (\cite[Proposition 8.1]{PS10}) with the first page being the Aomoto complex \cite[Corollary 8.3]{PS10}.
\bt  \label{PS} Given the pair $(X,\nu)$ and any field $\K$, we have that
\[ \alpha_i(X^\nu,\K)\leq \beta_i(X,\nu_\K). \]
\et

Next we recall the following theorem due to Papadima and Suciu \cite[Theorem 12.6]{PS10} specialised to our case.
\begin{theorem} \label{thm PS minimal}
Let $X$ be a connected finite minimal CW-complex with a fixed epimorphism $\nu\colon \pi_1(X)\twoheadrightarrow \Z$. Then the linearization of the equivariant cochain complex of the cover  $X^\nu$, with coefficients in $ \Z$ or a field $\K$, coincides with the Aomoto complex  $(H^{\ast}(X, \Z), \cdot \nu_\Z)$ or $(H^{\ast}(X, \K), \cdot \nu_\K)$, respectively.
\end{theorem}

As an application of this theorem, we have the following result. 

\bp \label{prop minimal}
If $X$ is a connected finite minimal CW complex  with a fixed epimorphism $\nu\colon \pi_1(X)\twoheadrightarrow \Z$ such that $\alpha_i(X^\nu,\C)=\beta_i(X,\nu_\C)$, then we have
\begin{center}
$\Delta'_i(X^\nu)(1)\neq 0$ and $ \Delta'_i(X^\nu)(1)\mid \tau_i(X, \nu_\Z), $
\end{center} 
where $ \Delta'_i(X^\nu)$ is obtained from $\Delta_i(X^\nu)$ by taking out $t-1$ factors. The same conclusion  also holds at homological degree $i-1$. 
\ep
\begin{proof}
When $X$ is a minimal CW complex, up to homotopy and under suitable basis we can assume $\partial_i^\Z(1)=0$ and
write $\partial_i^\Z$ in matrix form $(t-1)A_i$. Under the same basis, the above theorem shows that the map $H^i(X,\Z)\overset{\nu_\Z}{\to} H^{i+1}(X,\Z)$ can be identified with $A_i(1)$ (up to dual operation). 
Note that
$$\alpha_i(X^\nu,\C)=b_i(X)-\rank \partial_i^\Z -\rank\partial_{i-1}^\Z $$
and
$$\beta_i(X,\nu_\C)=b_i(X)-\rank A_i(1) -\rank A_{i-1}(1) .$$
Since $\rank \partial_i^\Z=\rank (t-1)A_i= \rank A_i$,  $\alpha_i(X^\nu,\C)=\beta_i(X,\nu_\C)$ 
implies
\begin{center}
 $ \rank A_i=\rank A_i(1)$ and $ \rank A_{i-1}=\rank A_{i-1}(1)$.
\end{center}
 Let  $r_i$ denote the rank of $A_i$. 
For a $(r_i\times r_i)$-minor of $A_i$, say $g$, it may happen $g(1)=0$. Hence $\Delta( A_i)$ taking values at 1 divides the greatest common divisor of all possible  $(r_i\times r_i)$-minors of $A_i(1)$, which coincides with $\tau_i(X, \nu_\Z).$ In particular, $\Delta(A_i)$ taking value at 1 is not 0.  Then the conclusion follows by noticing that $\Delta(A_i)=\Delta'_i(X^\nu)$. The claim for homological degree $i-1$ follows by the same argument.
\end{proof}

Hyperplane arrangement complement satisfies the assumptions in the above proposition. For more details, see Section \ref{sec hyperlane arr}.

\section{Cohomology jump loci and orbifold groups}

\subsection{Cohomology jump loci} \label{sec 3.1}
Let $X$ be a connected finite CW-complex with $\pi_{1}(X)=G$  and $\K$ be an algebraically closed field, e.g. $\C$, $\p$. The group of $\K$-valued  characters, $ \homo(G,\K^*)$, is the  moduli space of  rank one  $\K$-local system on $X$.
The cohomology jump loci $\sV^i_k(X,\K)$ of $X$ are defined as in Definition \ref{def jump loci}.
Cohomology jump loci  are closed sub-varieties of $\homo(G,\K^*)$ and homotopy invariants of $X$. 
For cohomological degree one, $\sV^1_k(X,\K)$  depends only on $\pi_1(X)$ (e.g. see \cite[Section 2.2]{Suc11}). So for any finitely presented group $G$, $\sV^1_k(G,\K)$ is well defined.

\br  \label{homology vs cohomology}
One can also define the homology jump loci of $X$ as follows
$$\mathcal{W}_i^k(X,\K)\coloneqq\lbrace \rho\in \homo(G,\K^*) \mid \dim_{\K} H_{i}(X, L_{\rho})\geq k \rbrace.$$ 
Let $\rho^{-1}$ denote the inverse of $\rho$ in   $\Hom(G,\K^*)$. Then  we have the following isomorphism between $\K$-vector spaces \cite[p. 50]{Dim04} 
$$ H^{i}(X, L_{\rho^{-1}}) \cong \Hom_{\K}(H_{i}(X,L_{\rho}),\K),$$  which gives  $$\V^{i}_k(X,\K)=\lbrace \rho\in  \Hom(G,\K^*) \mid \rho^{-1} \in \mathcal{W}_{i}^k(X,\K)\rbrace.$$
So $\V^i_k(X,\K)$ and $\mathcal{W}^k_i(X,\K)$ share the same information. 
\er

\br \label{degree 0}   If $X$ is a finite connected CW complex, then  $\V_1^0(X, \K) = \{\K_X\}$ consists of just one point, the trivial rank-one local system,  and $\V^0_k(X, \K) = \emptyset$ for $k>1$.
\er

Chomology jump loci can be viewed as  generalizations of $\alpha_*(X^\nu,\K)$.   In fact, consider an epimorphism $\nu\colon \pi_1(X) \to \Z$. It induces an embedding
$$\nu^*\colon \K^*\hookrightarrow\Hom(G,\K^*) .$$
By Remark \ref{homology vs cohomology} and Proposition \ref{prop betti number}, it is easy to see that $\alpha_i(X^\nu,\K)=k$ if and only if \begin{center}
 $\im \nu^* \subseteq \sV^i_k(X,\K) $ and $\im \nu^* \nsubseteq \sV^i_{k+1}(X,\K) $.
\end{center}
The corresponding generalizations of $\beta_*(X,\nu_\K)$ are the resonance varieties. For its definition and properties, see e.g. \cite{Suc11}.

\medskip

The following structure theorem for cohomology jump loci of complex smooth quasi-projective varieties  are contributed by many people and we name a few here: Green and Lazarsfeld \cite{GL91}, Simpson \cite{Sim93}, Arapura \cite{Ara97}, Dimca and Papadima \cite{DP14}, etc. It is finalized by Budur and Wang in \cite{BW15,BW20}.
\bt \label{structure theorem} \cite{BW15,BW20} If $X$ is a complex smooth quasi-projective variety, then $ \sV^i_k(X,\C)$ is a finite union of torsion translated sub-tori of $\Hom(G,\C^*)$.
\et

The following question abut the cohomology jump loci for arbitrary field coefficients are still open \cite[Section 3.5]{Suc01}:
\begin{que}
Does the structure theorem  hold for $ \sV^i_k(X,\K)$ of  a complex smooth quasi-projective variety $X$ with $\cha(\K)>0$?
\end{que}

  Theorem \ref{structure theorem} implies the following property for the Alexander polynomial associated to the pair $(X,\nu)$. 
\begin{prop} \label{torsion} Let $X$ be a complex smooth quasi-projective variety. For any epimorphism $\nu\colon \pi_1(X)\twoheadrightarrow \Z$, $\Delta_i(X,\nu)$ is of cyclotomic  type for any $i\geq 0$. Hence $\exp(\M_i(X^\nu))$ is a positive integer.
\end{prop}
\begin{proof}
 By \cite[Proposition 1.4]{BLW18},  the structure theorem for cohomology jump loci implies that all the roots of $\Delta_i(X^\nu,\C)$ are  roots of unity. Since $\Delta_i(X^\nu,\C)$ and $\Delta_i(X^\nu)$ only differ by multiplication with a non-zero constant integer, the claim follows.
\end{proof}

\br
Note that $\exp(\M_*(X^\nu))$  is not always an integer for general $X$. For example,  the integral
Alexander polynomial of the complement $X$  of the knot $4_1$ in $S^3$ \cite[p. 981]{SW1} is given by
$$\Delta_1(X^\nu)=t^2-3t+1, $$
 hence  $\exp(\M_1(X^\nu))= \frac{3+\sqrt{5}}{2}$.
\er

\subsection{Orbifold Groups} \label{section orbifold groups}
In this subsection, we compute the degree 1 cohomology jump loci and $L^2$-type invariants for orbifold groups.

Let $\Sigma_{g,r}$ be a Riemann surface of genus $g\geq 0$ and with $r\geq 0$ points removed.  To have $b_1(X)>0$,  we always assume $X\neq \mathbb{CP}^1, \C$.
 Consider  $\Sigma_{g,r}$ with $s$ marked points $\{q_1, \ldots, q_s\}$ and a weight vector ${\bm \mu}=(\mu_1,\cdots,\mu_s)$, where $\mu_i \in \Z_{>0}$, as in \cite{Dim07,ACM13,Suc14}.  The orbifold group $\Gamma=\pi_1^{\orb}(\Sigma_{g,r},{\bm \mu})$ associated to these data is  defined as
\[ \Gamma=\pi_{1}^{\orb}(\Sigma_{g,r},{\bm \mu})\coloneqq \pi_{1}(\Sigma_{g,r}\backslash \{q_{1}, \ldots, q_{s}\})/ \langle \gamma_{j}^{\mu_{j}}=1 \text{ for all } 1\leq j \leq s\rangle, \]
where $\gamma_{j}$ is a meridian of $q_{j}$. 
If $\mu_j=1$ for some $j$, we get the same orbifold group by omitting $q_j$ and $\mu_j$. So without loss of generality, we assume $\mu_j>1$ for all $1\leq j \leq s$.




\subsubsection{Non-compact case}
If $\Sigma_{g,r}$ is not compact (i.e., $r>0$), then $\pi_{1}(\Sigma_{g,r})$ is a free group with rank $n=2g+r-1$, hence
$ \Gamma\cong F_{n}\ast\mathbb{Z}_{\mu_{1}}\cdots\ast\mathbb{Z}_{\mu_{s}}. $
Then we have
$$\mathrm{Hom}(\Gamma,\K^*)\coloneqq \{(t_1, \ldots, t_n,\lambda_1,\ldots,\lambda_s)\in(\K^*)^{n+s}|\lambda_j^{\mu_j}=1\ \text{ for all }\ 1\leq j\leq s\}.$$ 

When $\cha(\K)=0$, we have the following short exact sequence as in \cite[Proposition 2.7]{ACM13}:
$$ 1\to (\K^*)^{n} \to \mathrm{Hom}(\Gamma,\K^*) \to \oplus_{j=1}^s C_{\mu_j}\to 1,$$
where $C_{\mu_j}$ is the cyclic multiplicative subgroup of $\K^*$ with order $\mu_j$. Hence $\mathrm{Hom}(\Gamma,\K^{*})$ has $\prod_{j=1}^s \mu_j$ connected components, with every connected component isomorphic to $(\K^*)^{n} $. 

On the other hand, when $\cha(\K)=p>0$, let $\mu'_j$ be the largest positive integer such that $\mu'_j \mid \mu_j$ ($\mu'_j$ could be 1) and $\gcd(p,\mu'_j)=1 $. 
Then we have the following short exact sequence:
$$ 1\to (\K^*)^{n} \to \mathrm{Hom}(\Gamma,\K^{*}) \to \oplus_{j=1}^{s} C_{\mu'_j}\to 1,$$
where $C_{\mu'_j}$ is the cyclic multiplicative subgroup of $\K^*$ with order $\mu'_j$. Hence $\mathrm{Hom}(\Gamma,\K^{*})$ has $\prod_{j=1}^{s} \mu'_j$ connected components, with every connected component isomorphic to $(\K^*)^{n} $. In particular, if $ \mu_j$ is a power of $p$ for all $1\leq j\leq s $ (i.e., $\prod_{j=1}^{s}\mu'_j=1$), $ \mathrm{Hom}(\Gamma,\K^{*})\cong (\K^*)^{n}. $

In both cases,  for any $\rho=(t_1, \ldots, t_n,\lambda_1,\ldots,\lambda_s) \in \Hom(\Gamma,\K^{*})$, we set $\ell_\K(\rho,{\bm \mu})$ as the number of trivial coordinates $\lambda_j=1$ such that $\cha(\K)=p\nmid \mu_j$. When $\rho$ is the trivial character, we simply write $\ell_\K({\bm \mu}). $ 


\bp \label{prop non-compact}
With the above notations,  we have $$\dim H^1(\Gamma, L_\rho)=\begin{cases}
n+s-\ell_\K(\rho,{\bm \mu})-1, & \mathrm{if}\ L_\rho\neq \K_\Gamma, \\
n+s-\ell_\K({\bm \mu}), & \mathrm{if}\ L_\rho =\K_\Gamma.
\end{cases} $$ 
Here $\K_\Gamma$ denotes the trivial character on $\Gamma$. Let $\Hom(\Gamma,\K^*)^0$ denote the connected component of $\Hom(\Gamma,\K^*) $, which contains  $\K_\Gamma$. Then we have
\[
\sV^1(\Gamma, \K)  = 
\begin{cases}
\Hom(\Gamma,\K^*),   &  \text{if } n > 1, \text{ or } n = 1 \text{ and } p \\
                   & \text{divides some } \mu_j, \\[5pt]
(\Hom(\Gamma,\K^*)\setminus \Hom(\Gamma,\K^*)^0) \cup \{\K_\Gamma\},   &  \text{if } n = 1, \prod_{j=1}^{s}\mu'_j > 1, \\
                   & \text{and } p \text{ does not divide any } \mu_j, \\[5pt]
\{\K_\Gamma\},             &  \text{if } n = 1, \prod_{j=1}^{s}\mu'_j= 1, \\
                   & \text{and } p \text{ does not divide any } \mu_j.
\end{cases}
\]
\ep
\begin{proof} Consider the Eilenberg–MacLane space $K(\pi_1,1)$ of $\Gamma$, say $Y$. The homology group $H_1(\Gamma, L_\rho)$ can be computed by the chain complex $C_*(Y, L_{\rho})$, see \cite[p. 50]{Dim04}. By Fox calculus    $\partial_1^{\rho}$ has the form of a $(n+s)\times s$ matrix as follows
\be \label{matrix non-compact} \begin{pmatrix}
\frac{\lambda_{1}^{\mu_{1}}-1}{\lambda_{1}-1} & 0 & \cdots & 0 \\
0 & \frac{\lambda_{2}^{\mu_{2}}-1}{\lambda_{2}-1} & \cdots & 0 \\
\vdots & \vdots & \vdots & \vdots \\
0 & 0 & \cdots & \frac{\lambda_{s}^{\mu_{s}}-1}{\lambda_{s}-1} \\
0 & 0 & \cdots & 0 \\
\vdots & \vdots & \vdots & \vdots \\
0 & 0 & \cdots & 0\\
\end{pmatrix}. \ee
Note that for any $1\leq j \leq s$, $\frac{\lambda_{j}^{\mu_{j}}-1}{\lambda_{j}-1}\neq 0$ in $\K $ if and only if $\lambda_j= 1$ and $p\nmid \mu_j$. 
Then we see that the above matrix has rank $\ell_\K(\rho, {\bm \mu})$.  If $L_\rho$ is not the constant sheaf, then the rank of $\partial_1^{\rho}$ is 1, hence
\[ \dim H_1(Y, L_{\rho})=n+s-\ell_\K(\rho,{\bm \mu})-1. \]
 On the other hand, if $L_\rho$ is the constant sheaf, then the rank of $\partial_1^{\rho}$ is 0, hence
\[ \dim H_1(Y, \K)=n+s-\ell_\K({\bm \mu}). \] Then the claim follows from Remark \ref{homology vs cohomology}.
\end{proof}

\subsubsection{Compact case} 
If $\Sigma_{g,r}$ is compact (i.e., $r=0$), then
\[ \Gamma\cong \langle x_{1},\ldots,x_{g},y_{1},\ldots,y_{g},\gamma_{1},\ldots,\gamma_{s}|\prod_{i=1}^{g}[x_{i},y_{i}]\prod_{j=1}^{s} \gamma_{j}=1,\gamma_{j}^{\mu_{j}}=1 \text{ for all } 1\leq j\leq s\rangle, \]
hence $$\Hom(\Gamma,\K^*)=\{(t_1, \ldots, t_{2g},\lambda_1,\ldots,\lambda_s)\in(\K^{*})^{2g+s}|\lambda_1\cdots\lambda_s=1, \lambda_j^{\mu_j}=1\ \mathrm{for}\ 1\leq j\leq s\}.$$

As in the non-compact case, when $\cha(\K)=p>0$, we set $\mu'_j$ to be the largest positive integer such that $\mu'_j \mid \mu_j$ and $\gcd(p,\mu'_j)=1 $. 
To unify the notations, we set $\mu'_j=\mu_j$ when $\cha(\K)=0$. 
 Then in both cases, we have the following short exact sequence:
$$ 1\to (\K^*)^{2g} \to \mathrm{Hom}(\Gamma,\K^{*}) \to  \big( \oplus_{j=1}^{s} C_{\mu'_j}\big)/C_{\mu'}\to 1,$$
where $\mu'\coloneqq\lcm(\mu'_1,\ldots,\mu'_{s})$ and the last term is the cokernel of the natural mapping $C_{\mu'}\to \oplus_{j=1}^{s} C_{\mu'_j}$. Hence $\mathrm{Hom}(\Gamma,\K^{*})$ has $\frac{\prod_{j=1}^{s} \mu'_j}{\mu'}$ connected components, with every connected component being isomorphic to $(\K^*)^{2g} $.

For any $\rho=(t_1, \ldots, t_{2g},\lambda_1,\ldots,\lambda_s) \in \Hom(\Gamma,\K^{*})$, we set $\ell_\K(\rho,{\bm \mu})$ as the number of trivial coordinates $\lambda_j=1$ such that $\cha(\K)=p\nmid \mu_j$. When $\rho$ is the trivial character, we simply write $\ell_\K({\bm \mu}).$ 


\bp \label{prop compact}
With the notations above, we have  $$\dim H^1(\Gamma, L_\rho)=\begin{cases}
2g+s-2-\ell_\K(\rho,{\bm \mu}), & \mathrm{if}\ L_\rho\neq \K_\Gamma, \\
2g+s-1-\ell_\K({\bm \mu}), & \mathrm{if}\ L_\rho =\K_\Gamma, \ell_\K({\bm \mu})<s \\
2g, & \mathrm{if}\ L_\rho =\K_\Gamma, \ell_\K({\bm \mu})=s.
\end{cases} $$ 
Let $\Hom(\Gamma,\K^*)^0$ denote the connected component of $\Hom(\Gamma,\K^*) $, which contains the constant sheaf $\K_\Gamma$. Then we have
\[
\sV^1(\Gamma, \K)  = 
\begin{cases}
\Hom(\Gamma,\K^*),   &  \text{if } g > 1, \text{ or } g = 1 \text{ and } p \\
                   & \text{divides some } \mu_j, \\[5pt]
(\Hom(\Gamma,\K^*)\setminus \Hom(\Gamma,\K^*)^0) \cup \{\K_\Gamma\},   &  \text{if } g = 1, \frac{\prod_{j=1}^{s} \mu'_j}{\mu'} > 1, \\
                   & \text{and } p \text{ does not divide any } \mu_j, \\[5pt]
\{\K_\Gamma\},             &  \text{if } g = 1, \frac{\prod_{j=1}^{s} \mu'_j}{\mu'} = 1, \\
                   & \text{and } p \text{ does not divide any } \mu_j.
\end{cases}
\]
\ep

\begin{proof}
 Consider the Eilenberg–MacLane space $K(\pi_1,1)$ of $\Gamma$, say $Y$. By Fox calculus  $\partial_1^{\rho}$ has the form of a $(2g+s)\times (s+1)$ matrix  as follows 
\be \label{matrix compact} \begin{pmatrix} 
\frac{\lambda_{1}^{\mu_{1}}-1}{\lambda_{1}-1} & 0 & \cdots & 0 & 1 \\
0 & \frac{\lambda_{2}^{\mu_{2}}-1}{\lambda_{2}-1} & \cdots & 0 & \lambda_{1} \\
\vdots & \vdots & \vdots & \vdots & \vdots \\
0 & 0 & \cdots & \frac{\lambda_{s}^{\mu_{s}}-1}{\lambda_{s}-1} & \prod_{j=1}^{s-1}\lambda_{j} \\
0& 0 & \cdots & 0 & y_{1}-1 \\
0& 0 & \cdots & 0 & 1-x_{1} \\
\vdots & \vdots & \vdots & \vdots & \vdots \\
0& 0 & \cdots & 0 & y_{g}-1 \\
0& 0 & \cdots & 0 & 1-x_{g} \\
\end{pmatrix}. \ee
 If $L_\rho $ is not the constant sheaf,  this matrix has rank $\ell_\K(\rho,{\bm \mu})+1$ and  $\partial_0^{\rho}$ has rank 1, hence
\[ \dim H_1(Y, L_{\rho})=2g+s-2-\ell_\K(\rho, {\bm \mu}). \]

While for the constant sheaf case, if $\ell_\K({\bm \mu})=s$, the above matrix has rank $s$; and if  $\ell_\K({\bm \mu})<s$ (i.e., there exists some $\mu_j$ such that $\cha(\K)=p\mid \mu_j$), the above matrix has rank $\ell_\K({\bm \mu})+1$. Then we have 
\[
\dim H_1(Y,\K)=\begin{cases}
2g+s-\ell_\K({\bm \mu})-1, & \mathrm{if}\ \ell_\K({\bm \mu})<s\\
2g, & \mathrm{if}\ \ell_\K({\bm \mu})=s.
\end{cases} \] 
Then the claim follows from Remark \ref{homology vs cohomology}.
\end{proof}

\br When $\cha(\K)=0$, all the computations in Proposition \ref{prop non-compact} and Proposition \ref{prop compact} are already done  by  Artal Bartolo, Cogolludo-Agust\'in and Matei in \cite[Section 2]{ACM13}. Proposition \ref{prop compact} also corrects a statement in \cite[Proposition 4]{Delz} for all characteristic, while the $\cha(\K)=0$ case is first corrected in loc. cit..
\er

\subsubsection{$L^2$-type invariants for orbifold group}
Let $\Gamma=\pi_1^{\orb}(\Sigma_{g,r},{\bm \mu})$ be the orbifold group associated to the data $(g,r,{\bm \mu}).$ As mentioned before, any epimorphism $\nu\colon \Gamma\twoheadrightarrow\Z$  gives a $\Z$-cover of the space $K(\Gamma,1)$. We set \begin{center}
$\Delta_1(\Gamma^\nu)\coloneqq\Delta_1(K(\Gamma,1)^\nu)$ and $\alpha_1(\Gamma^\nu, \K)\coloneqq \alpha_1(K(\Gamma,1)^\nu,\K)$. 
\end{center}


\bp  \label{prop orbifold group} With the above notations, for any epimorphism $\nu\colon\Gamma\twoheadrightarrow\Z$, we have
\begin{itemize}
\item[(1)]   $\alpha_1(\Gamma^\nu, \K)=2g+r-2+\#\{ j \mid \cha(\K)=p \text{ divides } \mu_j\}.$ 
\item[(2)] \[
\Delta_1(\Gamma^\nu)=\begin{cases}
\prod_{j=1}^s \mu_j, & \mathrm{if}\ r>0\\
(\prod_{j=1}^s \mu_j ) \cdot (t-1), & \mathrm{if}\ r=0.
\end{cases} \] 
\end{itemize}
\ep
\begin{proof}
(1) Note that for any $\rho \in \K^*$, the pulled character $\nu^* \rho$ always have $\ell_\K(\nu^*\rho, {\bm \mu})=\ell_\K({\bm \mu})$, since the torsion part of the abelianization of $\Gamma$ has to map to zero by $\nu$. Then the claim follows by computation.

(2) If $r>0$, $\lambda_j=1$ implies $ \dfrac{\lambda_j^{\mu_j}-1}{\lambda_j-1}=\mu_j$. Then the claim follows
from (\ref{matrix non-compact}). If $r=0$,  $x_i$ and $y_i$ in (\ref{matrix compact}) will be replaced by some power of $t$.  $\nu$ being epimorphism  implies that the greatest common divisor of the non-zero powers is $\pm 1$. Then the claim follows.
\end{proof}

\section{Proof of the main results}
The following well known result will play a crucial role in the proofs of Theorem \ref{thm jump loci} and Theorem \ref{thm L2}, see e.g. \cite[Lemma 3]{CKO}.
\bt \label{thm finitely generated}
Consider an orbifold map $f\colon X\to \Sigma$ of type $(g,r,{\bm \mu})$. 
 Let $F$ denote the generic fiber of $f$. Then we have a short exact sequence of groups
$$  \pi_1(F) \to \pi_1(X) \overset{f_*}{\to} \pi_1^{\orb}(\Sigma_{g,r},{\bm \mu}) \to 1,  $$
where the first map is induced by the inclusion from $F$ to $X$.
In particular the kernel of $f_*$ is finitely generated.
\et 

\subsection{Proof of Theorem \ref{thm jump loci}}
 By Remark \ref{homology vs cohomology}, we only need to prove the theorem for the homology version. By Theorem  \ref{thm finitely generated}, the orbifold map $f$ gives the following  group extension
\[ 1 \rightarrow K\rightarrow G\overset{f_*}{\to} \Gamma\rightarrow 1, \]
where $G=\pi_1(X) $, $\Gamma=\pi_1^{\orb}(\Sigma_{g,r},{\bm \mu})$ and the kernel $K$ of $f_*$ is a finitely generated normal subgroup of $G$.

Given a character $\rho \in \homo(\Gamma,\K^*)$,  it gives a rank one $\K$-local system $L_\rho$ of $\Gamma$.   Let $f^* L_{\rho}$ denote the rank one local system  pulling back to $G$. 
Then we have the Hochschild-Serre spectral sequence: 
\[ E^2_{s,t} = H_s(\Gamma,H_t(K,f^* L_{\rho}))\Rightarrow H_{s+t}(G,f^* L_{\rho}).\]

Next we explain how to review $H_t(K,f^* L_{\rho})$ as a $\Gamma$-module. 
Let $P_{\bullet}$ be a free $G$-resolution of $\Z$, then the $\Gamma$ action on $H_*(K, f^* L_{\rho})=H_*(P_{\bullet}\otimes_{ K} f^* L_{\rho})$ is induced by the tensor product of the $G$-action on $f^*L_\rho$ and $G$-action on $P_\bullet$.   Since $K $ is the kernel of $f_*$,  $f^*L_\rho$ taking restriction over $K$ is the constant sheaf. Then we have 
\[ H_*(K,f^* L_{\rho})=H_*(P_{\bullet}\otimes_{ K} f^* L_{\rho})\cong H_*(K,\Z)\otimes_\Z \K\]
as $\Gamma$-modules. 
The $\Gamma$-action on $H_*(K,\Z)\otimes_\Z \K $ is the tensor product of the $\Gamma$-conjugation on $H_*(K,\Z)$ and the $\Gamma$-action on $\K$ induced by $\rho$. To emphasize the second action, we denote the $\Gamma$-module $H_*(K,\Z)\otimes_\Z \K$ by $ H_1(K,\Z)\otimes_\Z L_{\rho}$.   
In particular, we have $E^2_{1,0}=H_1(\Gamma,L_{\rho})$ and $E^2_{0,1}=H_0(\Gamma, H_1(K,\Z)\otimes_\Z L_{\rho})$. Then the spectral sequence gives us the following short exact sequence
\[ 0\rightarrow E^2_{0,1}/\mathrm{im}d_2 \rightarrow H_1(G,f^* L_{\rho})\rightarrow H_1(\Gamma,L_{\rho})\rightarrow 0,\]
where $d_2\colon E^2_{2,0} \to E^2_{0,1}$  is the differential map on $E^2$-page. 
Hence $$\mathrm{dim}H_1(G,f^* L_{\rho})\geq \mathrm{dim}H_1(\Gamma,L_{\rho})$$ and the equality holds  if $ E^2_{0,1}=0$.  

Since $K$ is a finitely generated group,   $H_1(K,\Z)\otimes_\Z L_{\rho}$  is a finite dimensional $\K$-vector space, say denoted by $W$.  For any $\gamma\in\Gamma$,  let $c(\gamma)$ denote the corresponding linear transformation on $W$ induced by the $\Gamma$-conjugation on $H_1(K,\Z)$.  Note that the linear transformation $c(\gamma)$ is independent with the character $\rho$.  Therefore
\[ E^2_{0,1}= H_0(\Gamma, H_1(K,\Z)\otimes_\Z L_{\rho})\cong W/\langle \big(c(\gamma)\otimes\rho(\gamma)-\mathrm{Id}\big)(w)|w\in W,  \gamma\in\Gamma\rangle.  \]
Recall the moduli space $\Hom(\Gamma,\K^*)$ as in section \ref{section orbifold groups}. In the non-compact case, for any $\rho=(t_1,\cdots,t_{n},\lambda_1,\cdots,\lambda_s)\in \Hom(\Gamma,\K^*)$, $E^2_{0,1}=0$ if  we choose $t_i$ not  to be the inverse of the eigenvalues of $c(x_i)$ for some $1\leq i\leq n$. Here $x_i$ is the $i$-th generator of the group $\Gamma=F_{n}\ast\mathbb{Z}_{\mu_{1}}\cdots\ast\mathbb{Z}_{\mu_{s}}$. Then the claim follows since $c(x_i)$ has only finitely many eigenvalues and there are only finitely many choices for $\lambda_j$.   
The compact case follows by a similar proof. 

\subsection{Proof of Theorem \ref{thm L2}}
The formula for $\alpha_1(X^\nu,\K)$ follows from Theorem \ref{thm jump loci} and the computations in subsection \ref{section orbifold groups} directly. 
The rest part is devoted to the proof for $\M_1(X^\nu)$.

By Definition \ref{def orbifold effective}, one indeed gets the following sequence of epimorphisms:
$$ G\overset{f_*}{\to} \Gamma   \overset{\nu'}{\twoheadrightarrow} \Z.$$
The map $\nu'\colon \Gamma\twoheadrightarrow\Z$ gives a rank one local system $\mathscr{L}$ of $\Gamma$ with stalk $\Z[t^{\pm}]$.  Let $f^* \mathscr{L}$ denote the rank one local system pulling back to $G$.   
In particular, 
\begin{center}
    $H_1(\Gamma, \mathscr{L})\cong H_1(\Gamma^{\nu'},\Z)$ and $H_1(G, f^*\mathscr{L})\cong H_1(X^\nu,\Z)$.
\end{center}

Applying the Hochschild-Serre spectral sequence,  we get:
\[ E^2_{s,t}=H_s(\Gamma,  H_t(K, f^*\mathscr{L}))\Rightarrow H_{s+t}(G,  f^* \mathscr{L}), \]
where $K$ is the kernel of $f_*$. 
Similarly to the proof of Theorem \ref{thm jump loci},  we have $E^2_{1,0}=H_1(\Gamma, \mathscr{L})$, $E^2_{0,1}=H_0(\Gamma,  H_1(K, f^*\mathscr{L}))$ and 
a short exact sequence of finitely generated $\Z[t^\pm]$-modules
\[ 0\rightarrow E^2_{0,1}/\mathrm{im}d_2 \rightarrow H_1(G,f^* \mathscr{L})\rightarrow H_1(\Gamma, \mathscr{L})\rightarrow 0.\]
Note that  $H_0(\Gamma,  H_1(K, f^*\mathscr{L}))$ is isomorphic to 
\[H_1(K,\Z)\otimes \Z[t^\pm] /\langle\big(c(\gamma)\otimes \nu'(\gamma)-\mathrm{Id}\big)(x\otimes t^i)|x\otimes t^i \in H_1(K,\Z)\otimes \Z[t^\pm],  \gamma\in\Gamma \rangle \]
as $\Z[t^\pm]$-modules,  
where $c(\gamma)$ is the automorphism on $H_1(K,\Z)$ induced by the conjugation of $\gamma$ on $K$. Since $\nu' \colon\Gamma \to \Z $ is surjective, $H_0(\Gamma,  H_1(K, f^*\mathscr{L}))$ is a quotient abelian group of $H_1(K,\Z)$, hence $E^2_{0,1}$ is a finitely generated abelian group.  In fact, for any $x\in H_1(K,\Z) $ and $i\in\Z$, there exists some $\gamma \in \Gamma$ such that
$\nu'(\gamma)=i $, hence   $x\otimes t^i$ is equivalent to $c(\gamma)^{-1}(x) \otimes 1$ in $H_0(\Gamma,  H_1(K, f^*\mathscr{L}))$.

 $E^2_{0,1}$ being a finitely generated abelian group implies that the rank of $E^2_{0,1}$ as a $\Z[t^{\pm}]$-module is 0, hence so is $E^2_{0,1}/\im d_2$.  By Lemma \ref{lem commutative algebra} we have
\[ \Delta_1(G^{\nu})=\Delta_1(\Gamma^{\nu})\cdot \Delta(E^2_{0,1}/\mathrm{im}d_2).  \]
By Proposition \ref{torsion}, $ \Delta_1(G^{\nu})$ is of cyclotomic type, hence so is $\Delta(E^2_{0,1}/\mathrm{im}d_2)$ and 
$\Delta_1(\Gamma^{\nu})$. 
Then $\M( \Delta(E^2_{0,1}/\mathrm{im}d_2))$ only depends on the leading coefficient of $\Delta(E^2_{0,1}/\mathrm{im}d_2)$. 
This leading coefficient is indeed 1. 
Otherwise, there exists a prime number $p$ which divides the leading coefficient. Since $\Delta(E^2_{0,1}/\mathrm{im}d_2)$ is cyclotomic type, $p$ divides $\Delta(E^2_{0,1}/\mathrm{im}d_2)$.  By the definition of Alexander polynomial,  this implies that rank of $E^2_{0,1}/\mathrm{im}d_2 \otimes \Fp$ as a $\Fp [t^{\pm}]$-module is at least one. In particular, $E^2_{0,1}/\mathrm{im}d_2$ is not a finitely generated abelian group. This contradicts with the fact that $E^2_{0,1}$ is a finite generated abelian group.

Putting all together, we have that
\[ \M_1(X^\nu)=\M_1(\Gamma^{\nu})=\sum_{j=1}^{s}\log \mu_j , \]
where the last equality follows from Proposition \ref{prop orbifold group}.
\subsection{Some remarks}
We close this section with some remarks for orbifold maps. 

\br \label{rem finitely many eq} 
Given two orbifold maps $f_i\colon X\rightarrow\Sigma_i$ for $i=1,2$. We say that they are equivalent if there is an isomorphism $g\colon \Sigma_1\rightarrow\Sigma_2$ such that ${\bm \mu}_2(g(B_1))={\bm \mu}_1(B_1)$  and the following diagram is commutative
\[ \xymatrix{
 & X\ar[dl]_{f_1}\ar[dr]^{f_2} & \\
\Sigma_1\ar"2,3"^{g} & & \Sigma_2 .\\
}\]
For an orbifold map $f\colon X \to \Sigma$, the image of the induced embedding $ \Hom(\pi_1(\Sigma),\C^*) \hookrightarrow \Hom(\pi_1(X),\C^*)$ is called the shadow of $f$ as in \cite{ACM13}. 
Then  for any two hyperbolic orbifold maps $f_i\colon X\to \Sigma_i$ with $i=1,2$,  their shadows are either equal or intersect at most finitely many isolated torsion points, proved by Dimca, Papadima and Suciu in  \cite[Theorem 4.2]{DPS08} and \cite[Lemma 6.12]{DPS09}, see also \cite[Lemma 6.4]{ACM13}. In fact, they proved the claim with the further requirement $\frac{\prod_j=1^s \mu_j}{\gcd\{\mu_1,\cdots,\mu_s\}}>1$ if $g=1$ and $r=0$, but the proof works as long as $f$ is hyperbolic. 
\er

 \br \label{rem well-defined}  Assume that $\nu\colon \pi_1(X)\twoheadrightarrow \Z$ is orbifold effective. Note that for the last claim in Definition \ref{def orbifold effective} it is possible that $\nu$ factors through two different orbifold maps.

 If the associated orbifold map $f$ in Definition \ref{def orbifold effective} is hyperbolic, $f$ is unique under the equivalence relationship given in the above remark. In fact, $\nu$ and $f$  induces two injective maps \begin{center}
$ \nu^*_\C\colon \C^* \hookrightarrow \Hom(\pi_1(X), \C^*)$ and $f^*_\C\colon \Hom(\pi_1(\Sigma), \C^*) \hookrightarrow \Hom(\pi_1(X), \C^*) .$
\end{center}
In particular, $\im \nu_\C^* \subseteq \im f^*_\C$.   Remark \ref{rem finitely many eq} shows that $f$ is unique, hence it makes sense to call $\nu$ of type $(g,r, {\bm \mu})$.

On the other hand, when $f$ is null,  $\Sigma$ is either $\C^*$ or an Elliptic curve  and
${\bm \mu}$ is trivial. In particular, $\nu$ is of type $(0,2,1)$ or $(1,0,1)$, respectively.
\er
 
\br   \label{rem Hodge}
 We show that under certain Hodge structure assumptions  $\nu$ is always orbifold effective.
 \begin{itemize}
\item[(1)]  $\nu\colon \pi_1(X)\twoheadrightarrow\Z$ is induced by an algebraic map
$h\colon X\to \C^*$ if and only if $\nu$,  considered as an element in $H^1(X,\Z)$,  is of Hodge type $(1, 1)$, i.e., $\nu \in F^1 H^1(X,\C)\cap  \overline{F^1 H^1(X,\C)}$. Here $F$ stands for the Hodge filtration. This follows from Deligne’s theory of 1-motives (cf. \cite[(10.I.3)]{Del}).
Then $\nu$ is orbifold effective in this case. In fact, projectivising and resolving indeterminacy, we get a map $\overline{h} \colon \overline{X} \to \mathbb{CP}^1$. Using Stein factorization,  we get the following commutative diagram $$\xymatrix{
X \ar[d]^{f}  \ar@(dl,ul)"3,1"_(0.5){h} \ar[r] & \overline{X} \ar[d]^{ h'} \ar@(dr,ur)"3,2"^(0.5){\overline{h}}  \\
\im(f) \ar[r] \ar[d] & S \ar[d]^{h''} \\
\C^*\ar[r] & \mathbb{CP}^1,
}$$
where $h''$ is a finite map, $h'$ has connected fiber and  $f\coloneqq h'\vert_X\colon  X  \to \im(f)$. \cite[Lemma 2.2]{Dim07} shows that $f$ has connected generic fiber. Hence $f$ is an orbifold map and it makes $\nu$ orbifold effective.

In particular, if   $H^1(X,\Q)$ is a pure Hodge structure of type $(1, 1)$, then any epimorphism  $\nu\colon \pi_1(X)\twoheadrightarrow\Z$ is orbifold effective. 
This assumption is equivalent to that $X$ has a smooth compactification $\overline{X}$ such that $H^1(\overline{X},\Q)=0$. Typical examples are hyperplane arrangement complement, toric arrangement complement and the complement of some algebraic curves in $\mathbb{CP}^2$. See \cite[Example 2.3]{Dim07} for more examples.

\item[(2)]
 $\nu\colon \pi_1(X)\twoheadrightarrow\Z$ factors through the induced map on the fundamental group level by an algebraic map
$f\colon X\to E$ with $E$ being an Elliptic curve, if and only if,   $\nu$, considered as an element in $H^1(X,\Z)$, is contained in a dimension 2 weight 1 pure sub-Hodge structure of $ H^1(X,\Q)$.   This also follows from Deligne’s theory of 1-motives (cf. \cite[(10.I.3)]{Del}).
Using a similar proof as above, $\nu$ is also orbifold effective in this case. 

\end{itemize}
\er

\bex \label{ex non orbifold effective}
 In general $\nu$ is not necessarily always orbifold effective.  Consider $X=E\times \C^*$ with $E$ being an Elliptic curve. Then we have
$\pi_1(X,\Z)\cong \pi_1(E,\Z)\oplus \pi_1(\C^*,\Z)\cong \Z^3$. 
Take a corresponding basis for this direct sum. Say $\nu\colon \Z^3 \twoheadrightarrow \Z$ is given by three integers $(a,b,c)$.  Then $\nu$ is orbifold effective, if and only if, either $c=0$ or $a=0=b$. 
In fact, let $f\colon X\to \Sigma$ be an orbifold map. Then $f$ induces an injective map $H^1(\Sigma,\Q) \to H^1(X,\Q) $ of mixed Hodge structures and we have the following three possible cases:
\begin{itemize}
    \item If the mixed Hodge structure on $ H^1(\Sigma,\Q)$ has pure weight one, then $c=0$.
    \item If the mixed Hodge structure on $ H^1(\Sigma,\Q)$ has pure weight two, then $a=0=b$. 
    \item  $ H^1(\Sigma,\Q)$ is isomorphic to  $H^1(X,\Q) $. 
\end{itemize} 
The last case indeed can not happen, otherwise we have the following commutative diagram
    \[ \xymatrix{
  X\ar[r] \ar[d] & \Sigma \ar[d] \\
  \mathrm{Alb}_X \ar[r] & \mathrm{Alb}_\Sigma,\\
}\]  where the two vertical maps are Albanese maps. 
Then  the left vertical map and the bottom horizontal map both have to be isomorphisms. Hence the right vertical map is surjective. This is impossible since $\Sigma$ has complex dimension one less than $X$. 
\eex

\section{Hyperplane arrangement}\label{sec hyperlane arr}
In this section, we specialize to the case of hyperplane arrangements.
An arrangement of hyperplanes  is a finite collection of hyperplanes in a complex affine or projective space.
Let $\sA$ be a hyperplane arrangement in  $\CP^n$, defined by a product of degree 1 homogeneous polynomials 
$$ Q(\sA)\coloneqq \prod_{H\in \sA} f_H.$$
This also gives a central hyperplane arrangement in $\C^{n+1}$, denoted by $\overline{\sA}$.
Set
\begin{center}
$X(\sA)\coloneqq\CP^n -\sA$ and $M(\sA)\coloneqq \C^{n+1}\setminus \overline{\sA}$. 
\end{center}
Note that
 $M(\sA)=X(\sA)\times \C^*$, hence one can identify $\sV^1(X(\sA),\K)$ and $\sV^1(M(\sA),\K)$, see e.g. \cite[Section 3.3]{Suc14B}. 
When the context is clear, we simply write $X(\sA)$ as $X$.

\subsection{Some properties and examples}
We first list some facts for hyperplane arrangement complement $X$.
\begin{itemize}
\item[(1)] $X$ is  homotopy equivalent to a minimal CW complex \cite{DP03,Ran}, hence $H^i(X,\Z)$ is a free abelian group for any $0\leq i \leq n$. 
\item[(2)] The cohomology ring $H^*(X,\Z)$ is determined by the combinatorial data (i.e., the poset of intersections of hyperplanes) of $\sA$. \cite[Chapter 3]{Dim17}.  
\item[(3)] For any epimorphism $\nu \colon  \pi_1(X)\twoheadrightarrow \Z$,  $\alpha_i(X^\nu,\C)=\beta_i(X,\nu_\C)$ for any $i$. This follows from the tangent cone equality between $\sV^i_k(X,\C)$ and the corresponding resonance varieties of $X$, see e.g. \cite[Theorem 3.7]{CO}
\end{itemize}

 \medskip


Next we give an application of Proposition \ref{prop minimal}. For a positive integer $k>1$, let $\Phi_k(t)$ be the irreducible cyclotomic polynomial of primitive $k$-th root of unity. It is known that
\[ \Phi_k(1)=\begin{cases}
1, &  \text{if } k \text{ is not a power of any prime number}, \\
p, & \text{if } k=p^r  \text{ for some prime number } p \text{ and some positive intger } r.
\end{cases}\]
For any epimorphism $\nu \colon  \pi_1(X)\twoheadrightarrow \Z$, 
by Proposition \ref{torsion} $\Delta_i(X^\nu)$ is of cyclotomic type. Let $\Delta'_i(X^\nu)$ denote the polynomial obtained from $\Delta_i(X^\nu)$ by taking out $(t-1)$ factors and we write it down as
 $$ \Delta'_i(X^\nu)=c_i \prod_{k>1} \Phi_k(t)^{e_{i,k}} .$$

\bp \label{prop ha upper bound}
With the above assumptions and notations for the hyperplane arrangement complement $X$,  
for any degree $i>0$,
we have \be  \label{nice}
c_i (\prod_{p} p^{e_{i,p}+e_{i,p^2}+\cdots }) \mid \tau_i(X,\nu_\Z), \ee
where the product runs over all prime number $p>0$. 
Moreover, we have  $$\exp(\M_i(X^\nu))\mid \tau_i(X,\nu_\Z).$$
In particular, say $ \nu$ is orbifold effective of type $(0,r,{\bm \mu})$. 
Then we have
\[ \prod_{j=1}^s \mu_j\mid \tau_1(X,\nu_\Z).\]
\ep 
\begin{proof}  Since $\Delta_i(X^\nu)$ is of cyclotomic type, we get   $\M_i(X^\nu )=\log c_i$, where $c_i$ is the leading coefficient of $\Delta_i(X^\nu)$.  By the properties listed above, hyperplane arrangement complement satisfies the assumptions in Proposition \ref{prop minimal}, hence the first claim follows. It implies that $ c_i \mid \tau_i(X,\nu_\Z)$.  Combined with Theorem \ref{thm L2}, the third claim follows. 
\end{proof}

 One may compare (\ref{nice}) in the above proposition with the following result due to Cohen and Orlik \cite[Theorem 1.3]{CO} (also see \cite{PS10} for its generalization).  
\bt
Let $X$ be a hyperplane arrangement complement. Then for any epimorphism $\nu\colon \pi_1(X) \twoheadrightarrow \Z$ and any $\lambda\in \C^*$ with order being some power of a prime number $p$,
we have 
$$\dim H^i(X,L_\lambda) \leq \beta_i(X, \nu_{\p}).$$
\et

Next two examples show that both inequalities in Theorem \ref{PS} and Proposition \ref{prop ha upper bound} could happen.

\bex \label{ex 1}
Let $\mathcal{A}\coloneqq\{[x,y]\in \CP^1| x^d+y^d=0\}$. 
Consider the complement $M(\sA)$ of the corresponding central hyperplane arrangement $\overline{\sA}$ in $\C^2$. 
 Take $ \nu\colon \pi_1(M(\sA)) \to \Z$ as the epimorphism that sends the meridian for $H_\ell$ to an integer $n_\ell$ with $\gcd\{n_\ell \colon 1\leq \ell \leq d\}=1$.
It is easy to see that for any field $\K$  
\begin{center}
 $\alpha_1(M(\sA)^\nu,\K) =\begin{cases}
0, & \mathrm{if}\ \sum_{\ell=1}^d n_\ell\neq 0,\\
d-2 , & \mathrm{if}\ \sum_{\ell=1}^d n_\ell=0,
\end{cases} $
\end{center}
hence $\M_1(M(\sA)^\nu)=0$.

We have a basis $\{a_1, \ldots, a_d\}$ for $H^1(M(\sA),\Z)$, where $a_\ell$ corresponds to the $\ell$-th hyperplane for $1\leq \ell \leq d$, hence $\nu_\Z= \sum_{\ell=1}^d n_\ell a_\ell $.  $H^2(M(\sA),\Z)$ is generated by $\{ a_{ij}\coloneqq a_i \cup a_j\}_{1\leq i <j \leq d}$ with the relation $ a_{ij}= a_{1j} -a_{1i}$. 
Hence $H^2(M(\sA),\Z)$ has a basis $\{a_{1\ell}\}_{2\leq \ell \leq d}$. Let us write down the matrix of the map $H^1(M(\sA),Z) \overset{\nu_\Z}{\to} H^2(M(\sA),\Z)$ under these basis:
$$\begin{pmatrix}
n_2      & n_3      & n_4 & \cdots & n_d \\
-\sum_{i\neq 2} n_i   & n_3      & n_4 &  \cdots & n_d               \\
n_2      &   -\sum_{i\neq 3} n_i & n_4  &   \cdots & n_d               \\
\vdots & \vdots & \vdots & \vdots & \vdots       \\
n_2      & n_3      & n_4      & \cdots & -\sum_{i\neq d} n_i
\end{pmatrix}$$

When  $\sum_{\ell=1}^d n_\ell \neq 0$, this matrix has rank $d-1$ and $\tau_1(M(\sA),\nu_\Z)=(\sum_{\ell=1}^d n_\ell)^{d-2}$ (since $\gcd\{n_\ell \colon 1\leq \ell \leq d\}=1$). On the other hand, when $\sum_{\ell=1}^d n_\ell =0$, it has rank $1$ and the matrix can be changed to $ \begin{pmatrix}
1 & 0 & \ldots & 0\\
\vdots & \vdots & \vdots & 0\\
1 & 0 &\ldots & 0
\end{pmatrix}$ by elementary column operation since $\gcd\{n_\ell \colon 1\leq \ell \leq d\}=1$, hence $\tau_1(M(\sA),\nu_\Z)=1$. Then we have \begin{center}
 $\beta_1(M(\sA),\nu_\K) =\begin{cases}
d-2, & \mathrm{if}\ \sum_{\ell=1}^d n_\ell= 0,\\
d-2, & \mathrm{if} \sum_{\ell=1}^d n_\ell \neq 0 \text{ and } \cha(\K)=p \mid \sum_{\ell=1}^d n_\ell,\\
0, & \mathrm{otherwise}.
\end{cases} 
$
 \end{center} Hence if $\sum_{\ell=1}^d n_\ell \neq 0$ and $p$ divides $\sum_{\ell=1}^d n_\ell $, we have  $\alpha_1(M(\sA)^\nu,\C)<\beta_1(M(\sA),\nu_{\p})$ and
\begin{center}
 $\exp( \M_1(M(\sA)^\nu) )=1< (\sum_{\ell=1}^d n_\ell)^{d-2}=\tau_1(M(\sA),\nu_\Z).$
\end{center}

 Under suitable basis $\partial_1^\Z$ can be written as  form of matrix of size $d\times (d-1)$
$$\begin{pmatrix}
1-t^{n_1}      & 1-t^{n_2} &  \ldots & 1-t^{n_{d-1}} \\
t^{\sum_{\ell=1}^d n_\ell}-1      & 0 & \ldots & 0 \\
   0     & t^{\sum_{\ell=1}^d n_\ell}-1 &  \ldots & 0               \\
\vdots & \vdots & \ddots & \vdots       \\
      0    & 0      & \cdots & t^{\sum_{\ell=1}^d n_\ell}-1
\end{pmatrix}$$
One can get  
\begin{center}
 $\Delta_1(M(\sA)^\nu) =\begin{cases}
(t-1)(t^{\sum_{\ell=1}^d n_\ell}-1)^{d-2}, & \mathrm{if}\ \sum_{\ell=1}^d n_\ell\neq 0,\\
t-1 , & \mathrm{if}\ \sum_{\ell=1}^d n_\ell=0.
\end{cases} $
\end{center}
In particular, if $\sum_{\ell=1}^d n_\ell\neq 0$, $(1+t+\ldots + t^{(\sum_{\ell=1}^d n_\ell)-1})^{d-2} \vert_{t=1}= (\sum_{\ell=1}^d n_\ell)^{d-2}$, which coincides with $\tau_1(M(\sA),\nu_\Z)$.  
\eex

\bex \label{ex 2}
Let $\mathcal{B}$ be the deleted $B_3$-arrangement in $\CP^2$ with defining equations $$zxy(x-y)(x-z)(y-z)(x-y-z)(x-y+z).$$ Order the
hyperplanes as the factors of the defining polynomial. Let $X(\sA)$ be the complement of the arrangement.  
It was first discovered by Suciu \cite{Suc02} that $\V^1(X(\sA),\C)$  has a transalated component 
$$V\coloneqq \rho \otimes \{t, t^{-1},1,t^{-1},t,t^2,t^{-2} \mid t\in \C^*\} $$
with $\rho = (1,-1,-1,-1,1,1,1)\in (\C^*)^7$, e.g. see \cite[Example 6.16]{Dim17}. Here we take $\{z=0\}$ as the hyperplane at infinity.
This component is induced by the orbifold map $f\colon X\rightarrow \C^*$ as following
\[ f([x,y,z])=\frac{x(y-z)(x-y-z)^2}{y(x-z)(x-y+z)^2} \]
and $f$ is  of type $(0,2,2)$. 
Consider $\nu=(1,-1,0,-1,1,2,-2) \in H^1(X(\sA),\Z)$ induced by $f$.  Then we have 
\begin{center}
$\alpha_1(X(\sA)^\nu,\K) =\begin{cases}
1, & \mathrm{if}\ \cha(\K)=2,\\
0 , & \mathrm{if}\ \cha(\K)\neq 2,
\end{cases} $
and $\M_1(X(\sA)^\nu)=\log 2$.
\end{center}

On the other hand, the Aomoto complex for $\nu_\Z$ and $\nu_\K$ can be computed by the formula given in \cite[p.119]{Dim17}. A direct computation shows that 
\begin{center}
 $\beta_1(X(\sA),\nu_\K) =\begin{cases}
1, & \mathrm{if}\ \cha(\K)=2,\\
0 , & \mathrm{if}\ \cha(\K)\neq 2.
\end{cases} $ and $\tau_1(X(\sA),\nu_\Z)=4$.
\end{center}

Using the formula given in  \cite[Example 4.1]{Suc02} (be careful that in \cite[Example 4.1]{Suc02} hyperplanes are ordered different from ours), we get  $\Delta'_1(X(\sA)^\nu)=2(t+1)$, hence $\Delta'_1(X(\sA)^\nu)(1)=4$, which coincides with $ \tau_1(X(\sA),\nu_Z)$.
\eex

It would be interesting to have an example where $\Delta'_1(X(\sA)^\nu)(1) \neq \tau_1(X(\sA),\nu_\Z).$

\subsection{Multinet}
We first recall the definition of multinet.
\bd \label{def multinet} Let $\sA$ be an essential hyperplane arrangement in $\CP^2$.  A multinet on  $\sA$ is a partition of $\sA$ into $k\geq 3$ subsets $\sA_1,\cdots, \sA_k$, together with an assignment of multiplicities, $n\colon \sA \to \Z_{>0}$, and a subset $\mathfrak{X}$ of intersection points in $\sA$ satisfying the following conditions:
\begin{itemize}
\item[(a)] $\sum_{H\in \sA_i} n_H= \kappa$, independent of $i$;
\item[(b)] for each $H\in \sA_i$ and $H'\in \sA_j$ with $i\neq j$, $H\cap H' \in \mathfrak{X}$; 
\item[(c)] for each $x\in \mathfrak{X}$, the sum $\sum_{H\in \sA_i, x\in H} n_H$ is independent of $i$;
\item[(d)] for every $1\leq i \leq k$ and $H,H'\in \sA_i$, there is a sequence $H=H_0,H_1,\cdots,H_s=H'$ in $\sA_i$ such that $H_{j-1}\cap H_j \notin \mathfrak{X}$ for $1\leq j \leq s$.
\item[(e)] $\gcd(n_H\colon H\in \sA)$=1.
\end{itemize} Such multinet is called a $(k,\kappa)$-multinet.
\ed
By results of Pereira-Yuzvinsky \cite{PY} and Yuzvinsky \cite{Yuz}, $k=3$ or $4$ if $\vert \mathfrak{X}\vert>1$. It was conjectured by Yuzvinksy  that $k=4$ only happens for Hessian  arrangement.
As shown by Falk, Pereira and Yuzvinsky \cite{FY07,PY}, every $(k,\kappa)$-multinet determines a hyperbolic orbifold map $f\colon X(\sA) \to \Sigma_{0,k}\subset \CP^1$ by
$$f(x)\coloneqq [\prod_{H\in \sA_1} f_H^{n_H} (x), \prod_{H\in \sA_2} f_H^{n_H}(x)]. $$
Moreover, the positive dimension components of $\sV^1(X,\C)$ passing through the origin are combinatorially determined 
 as follows, see e.g. \cite[Section 3]{Suc14}.
\bt \cite{FY07} Let $\sA$ be an essential hyperplane arrangement in $\mathbb{CP}^2$ with complement $X$.  
Then every positive dimensional component of $\sV^1(X,\C)$ passing through the origin is obtained
as a pull-back as in Theorem \ref{thm Ara}(a) using either a local orbifold map $f_x$ coming from a intersection point $x$ in $\sA$ of multiplicity $m_x\geq 3$ or a hyperbolic map $f\colon  X\to \Sigma_{0,k}$ coming from a multinet structure on a sub-arrangement of $\sA$ with $k=3,4$.
\et

As $k=4$ conjecturally only  happens for Hessian arrangement, we consider a multinet on $\sA$ with $k=3$.  Under the following assumptions, we will show that the corresponding orbifold map has no multiple fiber.
\begin{assumption} \label{assumption} Given a $(3,\kappa)$-multinet on  $\sA$ in $\CP^2$, we assume that
\begin{itemize}
\item[(a)] $n_H=1$ for any $H\in \sA_1$ or $ \sA_2$.
\item[(b)] If $H,H'\in \sA_1$ and $H\cap H' \notin \mathfrak{X}$, then there is no third hyperplane in $\sA$ passing through $H\cap H'$. So is $\sA_2$.
\item[(c)] Consider  $\sA_3$ and $\mathfrak{X}$. If there exists a point $x\in \mathfrak{X}$ such that $x$ is contained in precisely one hyperplane $H\in \sA_3$, one replaces $\sA_3$ and  $\mathfrak{X}$ by     deleting  the hyperplane $H$ and the point $x$, respectively.  We assume that one can run this procedure for $\sA_3$ and $\mathfrak{X}$ until there is no hyperplane left in $\sA_3$.
\end{itemize} 
\end{assumption}
\bp With the above assumptions, the orbifold map corresponding to this multinet has no multiple fiber. In particular,  $\sV^1(X(\sA),\C)=\sV^1(M(\sA),\C)$  does not have translated two dimensional component corresponding to this orbifold map.
\ep 

\begin{proof}
Since $M(\sA)=X(\sA)\times \C^* $, we only need to prove the claim for $M(\sA)$. 
We take $\nu \in H^1(M(\sA),\Z) $, which maps the meridian of $H$ to $1$ for $H\in \sA_1\cup \sA_2$ and $-2 n_H$ for $H\in \sA_3$.
Then the claim follows if we can show that $\tau_1(M(\sA)^\nu)=1$.

Consider the map  $H^1(M(\sA),\Z) \overset{\nu_\Z}{\to} H^2(M(\sA),\Z)$. Let $\mathfrak{Y}_i$ denote the intersection points in the sub-arrangement $\sA_i$ but not contained in $\mathfrak{X}$.  Due to Brieskorn Decomposition \cite[Theorem 3.2]{Dim17}, we have $$H^2(M(\sA),\Z)\cong \oplus_{x\in \mathfrak{X}\cup \mathfrak{Y}_1\cup \mathfrak{Y}_2\cup \mathfrak{Y}_3} H^2(M(\sA_x),\Z),$$
where $\sA_x$ is the central line arrangement for the point $x$. For any $x\in \mathfrak{X}$, Definition \ref{def multinet}(c) and
 the choice of $\nu$ give $$\sum_{H\in \sA_1, x\in H} n_H+\sum_{H\in \sA_2, x\in H} n_H-\sum_{H\in \sA_3, x\in H} 2n_H=0.$$ Then   Definition \ref{def multinet}(c) together with Assumption \ref{assumption}(a) implies that one can find $H\in \sA_1$ with $n_H=1$ and $x\in H$.
Hence the computations in Example \ref{ex 1}  show that one can replace the sub-matrix corresponding to $\sA_i\times \mathfrak{X}$ by the column obtained from the incidence relation between the hyperplanes and $\mathfrak{X}$.   We have the following matrix
\begin{equation} \label{matrix0}
\bordermatrix{%
   &      \mathfrak{X} &    \mathfrak{Y}_1 &  \mathfrak{Y}_2 &  \mathfrak{Y}_3\cr
 \sA_1  &     C_1 &      B_1 & 0 & 0\cr  
  \sA_2 &     C_2 &        0 &  B_2    & 0\cr
  \sA_3 &     C_3 &     0 & 0 & B_3 
},
\end{equation}
Here $C_i$ is the $\vert \sA_i\vert \times \vert \mathfrak{X}\vert $ incidence matrix, where the entry $(H,x)$ is 1 precisely when  $x \in H$.

Next we compute the matrix $B_1$. Consider a graph, whose vertices correspond to the hyperplanes in $\sA_1$ and two vertices has an unordered edge if the corresponding two hyperplanes intersect at a point in $\mathfrak{Y}_1$. 
Definition \ref{def multinet}(d)  is equivalent to say that this graph is connected. Assumption \ref{assumption}(b) implies that  there is a one to one correspondence between the edges and $\mathfrak{Y}_1$.
Choose a spanning tree in this graph and consider the submatrix corresponding to the vertices and the edges in it. 
By Assumption \ref{assumption}(a),(b), we get a $\kappa\times (\kappa-1)$ matrix, where $\kappa=\vert \sA_1 \vert$ and every column in this matrix has $1$ and $-1$ both for precisely one entry and $0$ for the rest. By easy computations, we get that $B_1$ has rank $\kappa-1$ and has a $(\kappa-1)\times (\kappa-1)$ minor being $\pm 1$. 
 The same computation works for $B_2$. 
 
 For the $(3,\kappa)$ multinet, 
since  $\beta_1(M(\sA),\nu_\C)=\alpha_1(M(\sA)^\nu,\C)=-\chi(\Sigma_{0,3})=1$, the matrix (\ref{matrix0}) has rank $d-2$ with $d=\vert \sA \vert$, see \cite[Theorem 3.4]{LY}. Assumption \ref{assumption}(c) implies that $C_3$ can be changed to 
$ \begin{pmatrix}
\mathrm{Id} & 0 
\end{pmatrix}$ by elementary column operations, where $\mathrm{Id}$ is the identity matrix.   In particular, $C_3$ is full row rank.
 Putting all the computations for $C_3$, $B_1$ and $B_2$ together, we find a $(d-2)\times (d-2) $ minor being $\pm 1$. Then the proof is done. 
\end{proof}

\bex We give a list of examples which satisfies Assumption \ref{assumption}:
 Figure 1(a), Figure 2, Figure 3(a),(b) in \cite{FY07} with $A,B,C$ corresponding to $\sA_1,\sA_2,\sA_3$ and Figure 1 in  \cite{Suc14B}.
It is also easy to find examples which do not satisfy Assumption \ref{assumption}, see Figure 1(b) in \cite{FY07} or the monomial (alias Ceva)  arrangement (with $m\geq 3$) in \cite[Example 6.11]{Dim17}.  
\eex
\subsection{A question by Denham and Suciu}
Recall that $\sA$ is a hyperplane arrangement in  $\CP^n$, defined by a product of degree 1 homogeneous polynomials 
$$ Q(\sA)\coloneqq \prod_{1\leq \ell \leq d}^d f_\ell.$$
This also gives a central hyperplane arrangement in $\C^{n+1}$, denoted by $\overline{\sA}$.
Let $$F(\sA)\coloneqq \{x\in \C^{n+1}\mid Q(\sA)(x)=1\}$$ denote the Milnor fiber of the arrangement $\overline{\sA}$. Then $F(\sA)$ admits a finite cyclic cover to $X(\sA)$.

For the proof of Theorem \ref{thm our DS}, we use the same construction  introduced by Denham and Suciu in \cite{DS14}. 
Their  construction focuses entirely on `non-modular torsion`: that is, $p$-torsion in the homology of cyclic covers, where $p$ does not divide the order of the cover, hence part of results in \cite{DS14} are proved under this assumption. We  modify all the related results in \cite{DS14} without this assumption and list them as follows.

The next theorem, corresponding to \cite[Theorem 4.3]{DS14}, follows from Theorem \ref{thm L2} and Proposition \ref{prop large enough}.

\bt \label{thm DS 4.3}   Let $X$ be a complex smooth quasi-projective variety. Suppose that there is a 
orbifold fibration $f\colon  X \to \Sigma$  of type $(g,r,{\bm \mu})$ with $\prod_{j=1}^s \mu_j>1$.
 Take a prime number $p$ dividing $\prod_{j=1}^s \mu_j$. Then, for
all sufficiently large integer $N$, there exists a regular $N$-fold cyclic cover
$X^N \to X$ such that $H_1(X^N,\Z)$ has $p$-torsion. 
\et 
\begin{proof}
Let $\Gamma$ be the corresponding orbifold group. 
Take any epimorphism $\nu\colon \Gamma \twoheadrightarrow \Z$ and consider the covering space $X^\nu$ for the composed map $$\pi_1(X) \overset{f_*}{\twoheadrightarrow} \Gamma \overset{\nu}{\twoheadrightarrow} \Z.$$
Let $X^{\nu,N}$ denote the corresponding $N$-fold cover. Theorem \ref{thm L2} shows  $\alpha_1(X^{\nu},\C)<\alpha_1(X^\nu, \p)$. Then the claim follows from Proposition \ref{prop large enough}.
\end{proof}

For a technical reason, we need to allow multiplicities on the hyperplanes: if $\m\coloneqq(m_1,\ldots,m_d) \in \Z^{d}_{>0} $ is a positive lattice vector, then the pair $(\sA, \m)$ is called a multiarrangement. A defining
polynomial for the multiarrangement is the product
$$Q(\sA,\m)\coloneqq \prod_{\ell=1}^d f_\ell^{m_\ell}$$
and 
$$F(\sA,\m)\coloneqq \{x\in \C^{n+1}\mid Q(\sA,\m)(x)=1\} $$
is called the Milnor fiber of the multiarrangement $(\sA,\m)$. Set $m\coloneqq \sum_{\ell=1}^d m_\ell$. 
Then $F(\sA,\m) $ admits a regular $m$-fold cover of $X(\sA)$ given by the homomorphism \begin{center}
$\delta \colon \pi_1(X(\sA)) \to \Z/m\Z$, $ \gamma_\ell\mapsto m_\ell \mod m,$
\end{center}
where $\gamma_\ell$ is the meridian for $H_\ell$ in $X(\sA)$.

The following lemma corresponds to \cite[Proposition 6.7]{DS14}.
\bp \label{prop DS 6.7} Let $\sA$ be an hyperplane arrangement and let $X^N \rightarrow X$ be a connected regular $\mathbb{Z}/N\Z$ cover of $X=X(\sA)$ given by an epimorphism $\chi\colon \pi_1(X) \twoheadrightarrow \Z/N\Z $. Then  there exists infinitely many multiplicity vectors $\m\in \Z^{d}_{>0}$ for which the covering projection $F(\sA,\m) \to X $ factors through $X^N$
$$
\xymatrix{
F(\sA,\m) \ar[rd] \ar[r] & X^N \ar[d]\\
 & X
}
$$
Moreover, for any chosen prime $p>0$, we may choose $\m$ such that $\frac{m}{N}$ is not divisible by $p$, where $m=\sum_{\ell=1}^d m_\ell$.
\ep
\begin{proof}
For any $1\leq \ell \leq d$, say $\chi$ sends $\gamma_\ell$ to $\chi_\ell \mod N$, where $\chi_\ell \in [0, N-1]$ is an integer. In particular, $N$ divides $\sum_{\ell=1}^d \chi_\ell$.
Choose any multiplicity vector $(q_1,\ldots,q_d)\in \Z^{d}_{\geq 0}$ such that $m_\ell =\chi_\ell + N\cdot q_\ell >0$ and 
$\frac{m}{N}$ is not divisible by $p$, where  $m=\sum_{\ell=1}^d m_\ell$. 
Then it is clear $F(\sA,\m) \to X $ factors through $X^N$.
\end{proof}

\bex{\cite[Example 6.12]{DS14}} \label{ex key example} 
Fix an integer $\mu\geq 2$. Let $\sA_\mu$ be the deleted monomial arrangement,
where its defining equations in $\CP^2$ are $yz(x^\mu-y^\mu)(x^\mu-z^\mu)(y^\mu-z^\mu)  $. Ordering the
hyperplanes as the factors of the defining polynomial. Its projective complement $X(\sA_\mu)$ admits an orbifold map $X(\sA_\mu)\to \C^*$ given by
$$\dfrac{z^\mu(x^\mu -y^\mu)}{y^\mu(x^\mu-z^\mu)}, $$
which is of type $(0,2,\mu)$. 
Moreover, $\sV^1(M(\sA),\C)$ has a translated component 
$$V\coloneqq \rho \otimes \{t^\mu, t^{-\mu}, \overbrace{ t,\ldots, t}^\mu , \overbrace{ t^{-1},\ldots, t^{-1}}^\mu , \overbrace{  1,\ldots ,1 }^\mu \mid t\in \C^*\} $$
with $\rho = (1,\zeta^{-1},\ldots ,\zeta^{-1} ,1,\ldots, 1 , \zeta, \ldots, \zeta )$. Here $\zeta$ is a primitive $\mu$-th root of unity. 

Fix a prime number $p$. For large enough $N$,
we can choose the multiplicity vector as $$\m=(\mu, kN-\mu, \overbrace{ 1,\ldots, 1}^\mu , \overbrace{ N-1,\ldots, N-1}^\mu , \overbrace{  N,\ldots ,N }^\mu ) $$
with $m=(2\mu+k)N$ such that $p$ is coprime with $2\mu+k$. Then this multiplicity vector satisfies the requirement in Proposition \ref{prop DS 6.7}.
\eex

Next we go to the construction considered in \cite{DS14}.
For $m_1>2$, by \cite[Definition 7.8]{DS14} we consider the parallel connection of the multiarrangement $(\sA,\m) $  with a central line arrangement in $\C^2$, denoted by $\sA \circ_{H_{1}}\mathcal{P}_{m_1} $, where the line arrangement has exactly $m_1$ hyperplanes.
Up to change of coordinates, we can assume that $H_1$ for $\sA$ is defined by $\{x_1=0\}$ and the defining equations for the line arrangement $\mathcal{P}_{m_1}$ are $\prod_{j=0}^{m_1-1}  (y_1-j\cdot y_2)$. Then $\sA \circ_{H_{1}}\mathcal{P}_{m_1} $ can be viewed as a central multiarrangement in $\C^{n+2}$ with the following  defining equations by the coordinates $(x,y_2)\in \C^{n+2}$ 
$$   x_1\cdot\big(\prod_{\ell=2}^{d} f_\ell^{m_\ell}(x)\big)\cdot \prod_{j=1}^{m_1-1} (x_1 - j\cdot y_2). $$
In particular, $ X(\sA \circ_{H_{1}}\mathcal{P}_{m_1}) \cong X(\sA)\times \Sigma_{0, m_1}$, where $ \Sigma_{0, m_1}\coloneqq \mathbb{CP}^1 \setminus \{ m_1 \text{ points}\}$. 
If $m_1=1$, then $ \sA \circ_{H_{1}}\mathcal{P}_{m_1}$ is simply the original arrangement $\sA$.

The following proposition corresponds to \cite[Lemma 8.4]{DS14}.
\bp \label{prop DS 8.4} With the above notations, we have a pullback diagram between the respective Milnor covers,
\[ \xymatrix{
F(\mathcal{A},\m)\ar[r]\ar[d] & F(\sA \circ_{H_{1}}\mathcal{P}_{m_1})\ar[d] \\
X(\mathcal{A})\ar[r]^{j} & X(\sA \circ_{H_{1}}\mathcal{P}_{m_1}) \\
} \]
where $j$ is the map given by choosing a base point in  $\mathcal{P}_{m_1}$.
Moreover,  if $m_1\geq 3$ and $i\geq 1$, there exists a surjective homomorphism
$$H_{i}\big(F(\sA \circ_{H_{1}}\mathcal{P}_{m_1}),\Z\big)_{\tor} \twoheadrightarrow  H_{i-1}(F(\mathcal{A},\m),\Z)_{\tor}. $$ 
\ep 
\begin{proof}
The first claim is already proved in \cite[Lemma 8.4]{DS14} and we only need to prove the second one.

 $ \Sigma_{0,m_1}$ has a simple cell structure with a single $0$-dimensional  cell, say $a$, and $(m_1-1)$ many $1$-dimensional cell, say $\{b_1,\ldots,b_{m_1-1}\}$.   
The product structure $ X(\sA \circ_{H_{1}}\mathcal{P}_{m_1}) \cong X(\sA)\times \Sigma_{0, m_1}$ gives a cell structure for $X(\sA \circ_{H_{1}}\mathcal{P}_{m_1})$. 
Note that $ F(\mathcal{A},\m)$ and $ F(\sA \circ_{H_{1}}\mathcal{P}_{m_1})$ are $\Z/m\Z$-fold covers of $ X(\mathcal{A})$ and $ X(\sA \circ_{H_{1}}\mathcal{P}_{m_1})$, respectively.
By choosing fixed lifts of the cells of $ F(\mathcal{A},\m)$ and $ F(\sA \circ_{H_{1}}\mathcal{P}_{m_1})$ respectively, we obtain free basis for the chain complex  $C_{*}\big( F(\mathcal{A},\m),\Z\big)$ and $C_{*}\big( F(\sA \circ_{H_{1}}\mathcal{P}_{m_1}),\Z\big)$ as  $R$-modules, where  $R=\Z[t]/(t^m-1)\cong \Z[\Z/m\Z]$.  
The $R$-module map between finitely generated free $R$-modules
\begin{center}
$C_i(F(\sA,\m),\Z) \overset{\partial_i}{\to} C_{i-1}(F(\sA,\m),\Z)$
\end{center}
  can be written down as a matrix, say $A$, with entries in $\Z[t]/(t^m-1)$. 
By the product structure  $ X(\sA \circ_{H_{1}}\mathcal{P}_{m_1}) \cong X(\sA)\times \Sigma_{0, m_1}$, we have the following matrix for the map $C_{i+1}\big(F(\sA \circ_{H_{1}}\mathcal{P}_{m_1}),\Z\big)\to C_{i}\big(F(\sA \circ_{H_{1}}\mathcal{P}_{m_1}),\Z\big) $: 
\begin{center}
$\bordermatrix{%
 & (i-1,b_1)  & (i-1,b_2) & \cdots  & (i-1,b_{m_1-1}) & (i, a) \cr
(i,b_1)  &       A &      0 &      \cdots & 0 & (t-1)\mathrm{id}\cr
(i,b_2)  &       0 &      A &      \cdots & 0 & (t-1)\mathrm{id}\cr
 \vdots                    &  \vdots & \vdots &      \vdots & \vdots & (t-1)\mathrm{id}\cr
(i, b_{m_1-1})  &  0      & 0 &      \cdots & A & (t-1)\mathrm{id} \cr
(i+1, a)       &  0 & 0      &  \cdots  &  0 & *
}
$
\end{center}
Here for example $(i,b_1)$ stands for the basis $C_i(X(\sA), R)\otimes_R \langle b_1 \rangle $, where $\langle b_1 \rangle$  is the free $R$-module generated by $b_1$. 
When $m_1\geq 3$, by elementary  operations the above matrix becomes \begin{center}
$\begin{pmatrix} 
  A &      0 &      \cdots & 0 & (t-1)\mathrm{id}\\
  0 &      A &      \cdots & 0 & 0\\     
  \vdots &  \vdots & \vdots &      \vdots & \vdots \\
  0 & 0 &      \cdots & A & 0 \\
       0 & 0     &  \cdots  &  0 & *
\end{pmatrix}
$
\end{center}
Consider $\Z[t]/(t^m-1)\cong \Z^m$ as an abelian group. Then $A $ can be also viewed as a matrix (with different size) with entries in $\Z$ and the torsion part of  $H_{i-1}(F(\sA,\m),\Z) $ is determined by $A$.
Then the above computation gives the claim.
\end{proof}

Denham and Suciu iterated the parallel connections of the multiarrangement  $(\sA,\m)$ with a collection of line arrangements and obtained the following definition. 
\bd \cite[Definition 8.1]{DS14}
The polarization of a multiarrangement $(\mathcal{A},\m)$, denoted by $\mathcal{A}|| \m$, is the arrangement of $m=\sum_{\ell=1}^d m_\ell$ hyperplanes given by
\[ \mathcal{A}||\m=\mathcal{A}\circ_{H_{1}}\mathcal{P}_{m_1}\circ_{H_{2}}\mathcal{P}_{m_{2}}\circ_{H_{3}}\cdots\circ_{H_{d}}\mathcal{P}_{m_{d}}. \]
\ed
Note that the polarization of a multiarrangement is a reduced hyperplane arrangement with $m=\sum_{\ell=1}^d m_\ell$ many hyperplanes. Let $F(\mathcal{A}|| \m)$ denote the Minor fiber for the arrangement $(\mathcal{A}|| \m)$.

\begin{proof}[Proof of Theorem \ref{thm our DS}]
We take $\sA_\mu$ to be the arrangement considered in Example \ref{ex key example}  with $X$ being its complement. There exists an orbifold map $f$ of type $(0,2,\mu)$ for $X$. By choosing $\mu$ such that $p$ divides $\mu$,  for any integer $N$ big enough Theorem \ref{thm DS 4.3} gives $X^N$, a $\Z/N\Z$-cover of $X$, such that  $H_1(X^N,\Z)$ has non-trivial $p$-torsion. 
Take the multiplicity vector given in Example \ref{ex key example}. 
Since $F(\sA_\mu,\m)$ is a $\frac{m}{N}$-fold cover of $X^N$ and $p$ does not divides $\frac{m}{N}$, by \cite[Lemma 2.4]{DS14} $H_1(F(\sA_\mu, \m),\Z)$ also has non-trivial $p$-torsion. Using Proposition \ref{prop DS 8.4} repetitively, we get that $H_{2\mu+3}(F(\sA_\mu ||\m),\Z)$ has non-trivial $p$-torsion (since there are exactly $\mu$ many $m_\ell=1$). 

Then we take $\mathcal{B}= \sA_\mu ||\m$, which is a reduced hyperplane arrangement with $\vert \mathcal{B} \vert =m$.   
   For the additional requirement  $p$ dividing $\vert \mathcal{B} \vert$, we can  choose $N$ such that $p$ divides $N$, hence $p$ divides $m$. Then the claim follows.
\end{proof}

\br As we used  the same construction  as in \cite{DS14}, for any prime number $p\geq 2$,  the Milnor fiber of the central arrangement we construed can have  non-trivial $p$-torsion in homology at degree $2p+3$, same as in \cite[Corollary 8.8]{DS14}.  
Denham and Suciu asked if there is a hyperplane arrangement $\sA$ whose Milnor fiber $F(\sA)$ has torsion in $H_1(F(\sA),\Z)$ \cite[Question 8.10]{DS14}. Yoshinaga gave the first example where $H_1(F(\sA),\Z)$ has non-trivial $2$-torsion, see \cite{Yos20}.
\er


\end{document}